\documentclass[11pt]{article}

\usepackage{amsmath,amssymb,amsthm}
\usepackage{comment}
\usepackage[unicode,breaklinks=true,colorlinks=true]{hyperref}
\usepackage[dvipsnames]{xcolor}

\usepackage[top=1in, bottom=1in, left=1in, right=1in, marginparwidth=1in, marginparsep=0.1in]{geometry}

\numberwithin{equation}{section}
\newtheorem{theorem}{Theorem}[section]
\newtheorem{proposition}[theorem]{Proposition}
\newtheorem{lemma}[theorem]{Lemma}
\newtheorem{definition}[theorem]{Definition} 
\newtheorem{corollary}[theorem]{Corollary}

\theoremstyle{remark}
\newtheorem{remark}[theorem]{Remark}

\newcommand{\SVERAK}{\v Sver\'ak}

\definecolor{darkblue}{rgb}{0,0,0.7}

\newcommand{\rev}[1]{{#1}}
\newenvironment{revblock}{\par}{\par}

\newcommand{\al}{\alpha}

\newcommand{\de}{\delta}
\newcommand{\e}{\epsilon}
\newcommand{\ga}{{\gamma}}

\newcommand{\la}{\lambda}

\newcommand{\Om}{{\Omega}}

\newcommand{\td}{\tilde}

\newcommand{\De}{\Delta}

\newcommand{\R}{{\mathbb R }}
\newcommand{\N}{{\mathbb N}}

\newcommand{\pd}{{\partial}}
\newcommand{\nb}{{\nabla}}

\newcommand{\I}{\infty}

\renewcommand{\div}{\mathop{\mathrm{div}}}

\newcommand{\supp}{\mathop{\mathrm{supp}}}

\newcommand{\donothing}[1]{{}}

\newcommand{\EQ}[1]{\begin{equation}\begin{split} #1 \end{split}\end{equation}}
\newcommand{\EQN}[1]{\begin{equation*}\begin{split} #1 \end{split}\end{equation*}}

\makeatletter
\newcommand{\xRightarrow}[2][]{\ext@arrow 0359\Rightarrowfill@{#1}{#2}}
\makeatother

\newcommand{\loc}{\mathrm{loc}} 
\newcommand{\uloc}{\mathrm{uloc}}

\let\OLDthebibliography\thebibliography
\renewcommand\thebibliography[1]{
	\OLDthebibliography{#1}
	\setlength{\parskip}{1pt}
	\setlength{\itemsep}{1pt plus 0.3ex}
}

\begin{document}
	
	\title
	{	Time asymptotics, time regularity and separation rates for Navier-Stokes flows in supercritical solution classes
	}
	\author{Zachary Bradshaw\footnote{ Department of Mathematical Sciences, 
			University of Arkansas,
			Fayetteville, AR. \url{zb002@uark.edu}.
			 The research of Z.~Bradshaw was supported in part by the  NSF via grant DMS-2307097.} ~\& Joshua Hudson\footnote{Department of Mathematical Sciences,  
			 University of Arkansas,
			 Fayetteville, AR. \url{jh195@uark.edu}.}}
	\date{\today}

	 \maketitle
	 
	 	\begin{abstract}
	 	This paper extends the weak solution theory for the 3D Navier-Stokes equations of Barker, Seregin and \SVERAK~from a critical setting to a supercritical setting making sure to include a useful \textit{a priori} energy bound as well as a statement about stability under weak-star convergence. Two applications of the \textit{a priori} bound are then explored. The first provides a spatially local, short-time asymptotic expansion in the time variable starting at $t=0$ which, as a corollary, provides an upper bound on how fast hypothetical non-unique solutions to the Navier-Stokes equations can separate locally. The second establishes higher-order time regularity at a singular time and at spatial points positioned away from the singularity. This quantifies the degree to which the non-local nature of the pressure allows a far flung singularity to disrupt the time regularity at a regular point.
	 \end{abstract}
	 
	 \section{Introduction}
	 
	 The Navier-Stokes equations are a  system of PDEs which models the motion of a viscous, incompressible fluid.
	 The Navier-Stokes equations  for $n\geq 2$ are
	 \begin{equation}\label{eq.ns}\tag{NS}
	 	\partial_t u - \nu \Delta u +u\cdot \nb u +\nb \rev{\pi}= f;\qquad \nb \cdot u =0,
	 \end{equation}
	 where the velocity vector $u$  and scalar pressure $\rev{\pi}$ are unknowns and the viscosity coefficient $\nu$ and force $f$ are given. The system is augmented with an initial condition $u_0$ in an appropriate function space. In this paper, unless otherwise specified, $f=0$, $\nu=1$ and the domain is $\R^3\times (0,\I)$.
	 
	 A foundational mathematical treatment of \eqref{eq.ns} was provided by Leray in \cite{leray} where global weak solutions were constructed for finite energy data. Hopf later made contributions for bounded domains \cite{hopf}. Solutions resembling those constructed by Leray and Hopf are referred to as Leray solutions or Leray-Hopf solutions---we will refer to them as ``Leray solutions'' for brevity.  Although it has been nearly a century since Leray's original contribution, important questions remain open about \eqref{eq.ns}. For example, it is not known if Leray solutions with regular data can develop finite time singularities. It is also unknown if unforced Leray solutions  are unique.  In recent years, evidence has accumulated suggesting negative answers to these questions. In the direction of blow-up, Tao has constructed singular solutions for a nonlinear model replicating certain features of \eqref{eq.ns} \cite{Tao}. 
	 
	 Regarding uniqueness,  Buckmaster and Vicol have demonstrated non-uniqueness in a class of solutions which is weaker than the Leray class using convex integration \cite{BuckmasterVicol}. These solutions are not known to satisfy the global or local energy inequalities.  Under singular forcing,  non-uniqueness has been shown in the Leray class by Albritton, Bru\'e and Colombo \cite{AlBrCo}. Within the Leray class and with no forcing, a conjectural research program of Jia and \v Sver\' ak \cite{JS,JiaSverakIll}, as well as the numerical work of Guillod and \v Sver\' ak \cite{GS}, provide strong evidence for non-uniqueness. This program would first establish non-uniqueness in a class of suitable weak solutions with large data in the critical Lorentz space $L^{3,\I}$.  This space   coincides with the weak Lebesgue space $L^3_w$ and is close to but strictly larger than $L^3$, as is illustrated by the inclusion $|x|^{-1}\in L^{3,\I}\setminus L^3$. In lieu of a precise definition,  this example provides a good intuitive understanding of $L^{3,\I}$ in comparison to nearby spaces. Very recently, and building upon ideas of Palasek \cite{Pa}, Coiculescu and Palasek have constructed non-unique mild solutions with data in $BMO^{-1}$ with no external forcing \cite{CoPa}. Note that $L^{3,\I}\subset BMO^{-1}\cap L^2_\loc$ while $BMO^{-1}$ is not a subset of $L^2_\loc$.  So, the present result  of Coiculescu and Palasek  cannot be extended to establish non-uniqueness in the Leray class via a procedure like that outlined in \cite{JiaSverakIll}, but as they remark, it may be possible to refine their argument to achieve this. Ultimately, there is a wealth of evidence suggesting non-uniqueness in the Leray class for initial data in $L^{3,\I}\cap L^2$.

	 Solutions to \eqref{eq.ns} have a natural scaling which is reflected in the preceding discussion:~If $u$ is a distributional solution to  \eqref{eq.ns}   with associated pressure $\rev{\pi}$ and initial data $u_0$, then, for any $\lambda >0$, $
	 u^{\lambda}(x,t):=\lambda u(\lambda x,\lambda^2t)$
	 is also a solution with associated pressure 
	 $\rev{\pi^{\lambda}(x,t):=\lambda^2 \pi(\lambda x,\lambda^2t)}$
	 and initial data 
	 $u_0^{\lambda}(x):=\lambda u_0(\lambda x)$. 
	 If $X$ is an appropriate function space, then it can be classified according to this scaling via the following identity: $\|u_0\|_X = \la^\al \|u_0^\la \|_X$.  A space is \textit{supercritical} if $\rev{\al>0}$.  An example of such a space is $L^2$ and solution class is that of Leray.  These solutions are not known to be smooth or unique.
	 A space is  \textit{subcritical} if   $\rev{\al<0}$. Examples of subcritical spaces are $L^p$ for $p\in (3,\I]$. Local regularity and well-posedness  are known in subcritical spaces. 
	 A space is critical if $\al=0$, in which case $\|\cdot\|_X$ is scaling invariant.  Local regularity and well-posedness are known in some, but not all, critical spaces. For example, if $u$ is a suitable weak solution and 
	 \EQ{\label{ineq.ESS}
	 	\sup_{0<t<T} \|u(t)\|_{L^3} <\I,
	 }
	 where the above quantity can be any size,
	 then $u$ is smooth on $(0,T)$ \cite{ESS}. Another result in this direction says that, if $u_0\in L^3$ and is possibly large, then there exists a unique smooth solution, at least for a finite period of time \cite{Kato}. The proofs of these results use  that $C_c^\I$ is dense in $L^3$.
	 Critical spaces where this closure property fails   are sometimes referred to as \textit{ultracritical}.  An example of such a space is $L^{3,\I}$.   An analogue of \eqref{ineq.ESS}   in $L^{3,\I}$ is not known to imply regularity and uniqueness is not expected for large $L^{3,\I}$ data.  This reflects the fact that ultracritical spaces can generally be viewed as   \textit{borderline  spaces for well-posedness and regularity results}.

	 A solution is self-similar if $u=u^{\la}$ for all $\la>0$ and is discretely self-similar if this only holds for some $\la$.  In 3D,   self-similar data is $-1$-homogeneous and, therefore, can be in $L^{3,\I}$ but  not $L^3$.  As they are scaling invariant, self-similar solutions can be viewed as \textit{borderline  solutions for well-posedness and regularity results}.  
	 Due to their special structure,  they are useful  tools compared to general elements of $L^{3,\I}$   to study these issues.
	 This approach has been taken in the analysis of \eqref{eq.ns} both classically and recently. 
	 In the 1930s, Leray identified the backward self-similar case  as a possible blow-up scenario. This was ruled out in \cite{NeRuSv,Tsai1} but the possibility of discretely self-similar blow-up remains open.  Indeed, Tao's example of blow-up for a modified model has a backward discretely self-similar quality \cite{Tao}.  
	 More recently Albritton, Bru\'e and Colombo's proof of non-uniqueness for the forced Navier-Stokes equations involved self-similar solutions, as does the work of \SVERAK~and coauthors on the non-uniqueness problem for the unforced Navier-Stokes equations \cite{AlBrCo,GS,JS}. 
	 The existence of self-similar solutions for large data has been studied extensively, starting with the work of Jia and \SVERAK~\cite{JS} and continuing in \cite{AB,BT1,ChaeWolf,LR2,Tsai-DSSI}.

	 Because self-similar data do not belong to $L^2$, theories of weak solutions in other spaces are required for their analysis. In the first construction of self-similar solutions for large data \cite{JS}, the class of local Leray solutions of Lemari\'e-Rieusset was used \cite{LR}. 
		An incomplete list of subsequent work on solutions in this vein is \cite{KS,KwTs,JS,Basson,BT8,MMP2,FDLR,LR-Morrey}.
	 Building off of an idea of C.~Calder\'on \cite{Calderon}, revised solution classes were introduced for $L^3$ data in \cite{SeSv},  for $L^{3,\I}$ data in \cite{BaSeSv} and for Besov space data in \cite{Barker,AB}. This paper is primarily interested in the $L^{3,\I}$ scenario,  which we refer to as $L^{3,\I}$-weak solutions. To avoid the technicalities of a precise definition at this point, it is sufficient to think of these as global, infinite energy solutions which admit self-similar data and   have nice properties compared to other   classes of weak solutions.  Notably, these classes have stronger properties than Leray weak solutions, reflecting the fact that the initial data is locally less singular than things can be in $L^2$.  This makes them natural spaces to study properties of weak solutions in more restrictive, namely \textit{critical}, classes than the Leray class. For example,  these solution theories have applications to weak-star stability \cite{SeSv,Barker,BaSeSv,AB}, concentration results at a singularity \cite{AB}, the structure of the singular set \cite{AB,Popkin}  and time-asymptotics at $t=0$ \cite{BP2}. We emphasize that these applications are mostly restricted to a \textit{critical} class of weak solutions (an exception is the result of Popkin \cite{Popkin}).  The related splitting idea of Calder\'on \cite{Calderon} has also been used in a number of applications, e.g., \cite{LR,Barker2,BW}.
	 
	 The weak solution theories of \cite{SeSv,BaSeSv,AB} and their applications are mostly confined to {critical} scenarios. But, blow-up or non-uniqueness could, hypothetically, emerge in strictly supercritical contexts. There is an algebraic gap between the worst case singular scenario allowed under the Leray scaling (which are only marginally better than $|x|^{-3/2}$) and the singularities in the initial data in \cite{SeSv,BaSeSv,AB} (which scale like $|x|^{-1}$). It would be beneficial to prove theorems in this intermediate range.  Motivated by this, the first goal of this paper is to extend the theory of   $L^{3,\I}$-weak solutions to $L^{p,\I}$ where $2<p<3$. Unlike $L^p$, these spaces contain \rev{homogeneous} initial data. A complementary work in this direction has already been carried out by Popkin in supercritical Besov spaces   \cite{Popkin}. \rev{Specifically, Popkin constructs global weak solutions for divergence-free data in $\dot B^{s}_{q,\I}$, where $q>2$ and $-1+\frac 2 q<s<\min\{-1+\frac 3 q,0\}$, by splitting the data into a subcritical piece, which generates a global mild solution, and a piece in $L^2$, which is treated by energy methods. As an application he bounds the number of singular points of Leray solutions satisfying a Type~I bound in Besov spaces. These Besov spaces contain the Lorentz spaces considered here due to the embedding $L^{p,\I}\subset \dot B^{3/q-3/p}_{q,\I}$ for $q>p$. The index $s=3/q-3/p$ matches Popkin's range when $p<q<\frac p {3-p}$,   which is non-empty for $2<p<3$.} Because Popkin works in the full scale of Besov spaces a certain ``decay'' estimate at $t=0$ is unavailable. This estimate is necessary  for our applications, a fact which leads us to develop a complimentary theory of $L^{p,\I}$-weak solutions in full detail from scratch. Doing so also allows us to establish stability under weak-star convergence in a more general setting than has been considered previously.   
	 
	  In the remainder of this section we first review $L^{3,\I}$-weak solutions. We then present our weak solution theory and introduce two applications, one of which quantifies the severity of non-uniqueness by way of a time-asymptotic expansion, thereby extending earlier work from \cite{BP2}, while the other establishes time regularity at regular points at a singular time.

	 \subsection{Review of $L^{3,\I}$-weak solutions}

	We begin by defining  $L^{3,\I}$-weak solutions as in \cite{BaSeSv}.

	 \begin{definition}[$L^{3,\I}$-weak solutions]\label{def.weak3} Let $T>0$ be finite. Assume $u_0\in L^{3,\I}$ is divergence free. We say that $u$ and an associated pressure $\rev{\pi}$ comprise an $L^{3,\I}$-weak solution if \begin{enumerate}
	 		\item $(u,\rev{\pi})$ satisfies \eqref{eq.ns} distributionally,
	 		\item $u$ satisfies the local energy inequality of Scheffer \cite{VS76b} and Caffarelli, Kohn and Nirenberg \cite{CKN}, i.e.~for all non-negative $\phi\in C_c^\I (\R^3\times (0,T])$ and $0<t<T$, we have
	 		\EQ{\label{ineq:CKN-LEI}
	 			&\int \phi(x,t) |u(x,t)|^2 \,dx + 2\int_0^t \int |\nb u|^2\phi\,dx\,dt \\&\leq \int_0^t \int |u|^2(\pd_t \phi + \De\phi )\,dx\,dt +\int_0^t \int (|u|^2+2\rev{\pi})(u\cdot \nb\phi)\,dx\,dt,
	 		}
  	 		\item for every $\rev{\psi}\in L^2$, the following function is continuous on $[0,T]$\rev{, where $\td u$ is as in (4) and $u(\cdot,0)=u_0$},
  	 		\[
  			t\to \int \rev{\psi(x)}\cdot \rev{\td u}(x,t)\,dx,
  	 		\]
	 		\item $\td u :=u-e^{t\Delta}u_0$ satisfies, for all $t\in (0,T)$,
	 		\EQ{\label{ineq.BSSbound}
	 			\sup_{0<s<t}\| \td u \|^2_{L^2} (s) + \int_0^t \| \nb \td u\|_{L^2}^2(s)\,ds   <\I,
	 		}
	 		and
	 		\EQ{ \label{ineq:energyIneq}
	 			\| \td u\|_{L^2}^2(t) +2\int_0^t \int|\nb \td u|^2\,dx\,ds\leq 2\int_0^t \int ( e^{s\Delta}u_0 \otimes \td u + e^{s\Delta}u_0 \otimes e^{s\Delta}u_0) : \nb \td u\,dx\,ds.
	 		}
	 	\end{enumerate}
	 \end{definition}

  In \cite{BaSeSv}, it is proven that $L^{3,\I}$-weak solutions exist for any divergence free $u_0\in L^{3,\I}$. Stability under weak-star perturbations is also proven. An important observation in \cite{BaSeSv} is that the nonlinear part $\td u$ of a  $L^{3,\I}$-weak solution satisfies a dimensionless energy estimate, namely
	 \EQ{\label{ineq.BSSdecay}
	 	\sup_{0<s<t}\| \td u \|_{L^2} (s) +\bigg(\int_0^t \| \nb \td u\|_{L^2}^2(s)\,ds \bigg)^\frac 1 2 \lesssim_{\rev{\|u_0\|_{L^{3,\I}}}} t^{\frac 1 4}.
	 }
	 \rev{The implicit constant in \eqref{ineq.BSSdecay} depends on $u_0$ only through $\|u_0\|_{L^{3,\I}}$.} We emphasize that the energy associated with $\td u$ vanishes at $t=0$. A generalization of this property will play an important role in the present paper. It appeared earlier in the \textit{a priori} estimates of the weak discretely self-similar solutions constructed in \cite{BT1} as well as \cite{SeSv}, which is the precursor to \cite{BaSeSv}. 
	 It is used in the Calder\'{o}n-type splitting (see \cite{Calderon}) construction in \cite{BaSeSv} to deplete a time singularity in an integral estimate.
	 In \cite{AB}, it is established in Besov spaces with $e^{t\Delta}u_0$ replaced by higher Picard iterates.
	 As pointed out in \cite{BaSeSv}, the inequality \eqref{ineq.BSSdecay} is interesting in its own right as it can be viewed as an \textbf{estimate on the energy separation rate} of two non-unique   $L^{3,\I}$-weak solutions with the same data since, denoting two such solutions  by $u_1$ and $u_2$, we have
	 \[
	 \| u_1 - u_2\|_{L^2}(t)\lesssim \| \td u_1\|_{L^2}(t)+\|\td u_2\|_{L^2}(t)\lesssim t^{\frac 1 4}.
	 \]
	  In practice  it may be that supercritical singularities emerge from smooth solutions, at which point non-uniqueness could hypothetically occur. The preceding estimate, which only applies to critical singularities, does not limit the rate at which the error energy can grow in such a scenario. One motivation for our paper is to extend this  to supercritical non-uniqueness scenarios by way of a generalized notion of $L^{p,\I}$-weak solutions.

	 	 \subsection{ $L^{p,\I}$-weak solutions}
	 Our objective is  to extend the definition and main results from \cite{BaSeSv}.  
	 We run into some difficulty naively adapting the definition from \cite{BaSeSv}, however, because the term
	 \[
	 2\int_0^t \int ( e^{s\Delta}u_0 \otimes \td u + e^{s\Delta}u_0 \otimes e^{s\Delta}u_0) : \nb \td u\,dx\,ds,
	 \]
	 may not converge when $p$ gets close to $2$. In particular, if $\frac {12} 5<p$, then it is finite but this is not clear when $p\leq \frac {12}5$. Note that it is possible for the left-hand side of \eqref{ineq:energyIneq} to be finite while the above term is infinite. Therefore, a version of weak solutions can still be  formulated  for a modified class of solutions. \rev{The inequality \eqref{ineq:energyIneq} plays a second role in Definition \ref{def.weak3}, however: letting $t\to 0^+$ in it shows that $\|\td u(t)\|_{L^2}\to 0$, that is, the nonlinear part of the solution vanishes strongly in $L^2$ at the initial time. This property is not known to follow from the remaining conditions and it is needed in our arguments (see the proof of Lemma \ref{lemma.energy}), so we impose it directly. We define this class now with the understanding that $2<p\leq 3$, although our main interest is in $2<p<3$.}
	 
\begin{revblock}
	 	 \begin{definition}[$L^{p,\I}$-weak solutions]\label{def.weak} Let $T>0$ be finite. Assume $u_0\in L^{p,\I}$ is divergence free and let $\td u :=u-e^{t\Delta}u_0$, where we use the convention $u(\cdot,0)=u_0$. We say that $u$ and an associated pressure $\pi$ comprise an $L^{p,\I}$-weak solution on $[0,T]$ with data $u_0$ if 
	 	 \begin{enumerate}
	 		\item $\pi\in L^{3/2}_\loc(\R^3\times (0,T))$ and $\td u$ satisfies, for all $t\in (0,T)$,
	 		\EQ{\label{ineq.BSSbound.p}
	 			\sup_{0<s<t}\| \td u \|^2_{L^2} (s) + \int_0^t \| \nb \td u\|_{L^2}^2(s)\,ds   <\I,
	 		}
	 		\item $(u,\pi)$ satisfies \eqref{eq.ns} distributionally in $\R^3\times (0,T)$,
	 		\item $u$ satisfies the local energy inequality of Scheffer \cite{VS76b} and Caffarelli, Kohn and Nirenberg \cite{CKN}, i.e.~for all non-negative $\phi\in C_c^\I (\R^3\times (0,T])$ and $0<t<T$, we have
	 		\EQ{\label{ineq:CKN-LEI.p}
	 			&\int \phi(x,t) |u(x,t)|^2 \,dx + 2\int_0^t \int |\nb u|^2\phi\,dx\,dt \\&\leq \int_0^t \int |u|^2(\pd_t \phi + \De\phi )\,dx\,dt +\int_0^t \int (|u|^2+2\pi)(u\cdot \nb\phi)\,dx\,dt,
	 		}
	 		\item for every  $\psi\in L^2$, the following function is continuous on $[0,T]$,
	 		\[
	 		t\to \int \psi(x)\cdot \td u(x,t)\,dx,
	 		\]
	 		\item $\| \td u\|_{L^2}(t)\to 0$ as $t\to 0^+$.
	 	\end{enumerate}
	 \end{definition}

\medskip\noindent\textbf{Remarks on Definition \ref{def.weak}:}
\begin{enumerate}
\item Condition (5) replaces the energy inequality \eqref{ineq:energyIneq} of Definition \ref{def.weak3}. It cannot simply be dropped, which we acknowledge was pointed out to us by a referee: weak continuity at $t=0$ only gives $0=\|\td u(0)\|_{L^2}\leq \liminf_{t\to0^+}\|\td u(t)\|_{L^2}$, which is the wrong direction for our purposes. For instance, when $u_0=0$, conditions (1)--(4) describe a suitable weak solution in the energy class which attains its zero initial data only weakly, and we are not aware of an argument showing that such a solution must vanish identically. Condition (5) is used in the proof of Lemma \ref{lemma.energy}, on which Theorem \ref{thrm.aprioribound} is based, and Theorem \ref{thrm.aprioribound} in turn quantifies the rate in (5).
\item Condition (5) may equivalently be replaced by an energy inequality of the type in Definition \ref{def.weak3}, written in terms of the Calder\'on-type splitting of Lemma \ref{lemma.split}. Namely, decompose $u_0=\bar u_0^N+\td u_0^N\in L^4+L^2$ as in Lemma \ref{lemma.split}, and let $\bar V^N=e^{t\Delta}\bar u_0^N$ and $w^N=u-\bar V^N$. Under (1)--(4), condition (5) is equivalent to the requirement that, for a sequence of values $N\to\I$ and all $t\in(0,T)$,
\[
\|w^N(t)\|_{L^2}^2+2\int_0^t\|\nb w^N\|_{L^2}^2\,ds\leq \|\td u_0^N\|_{L^2}^2+2\int_0^t\int (\bar V^N\otimes w^N+\bar V^N\otimes\bar V^N):\nb w^N\,dx\,ds.
\]
In contrast to the right-hand side of \eqref{ineq:energyIneq}, the right-hand side here is finite for every $2<p<3$: by Ladyzhenskaya's inequality and $\|\bar V^N(s)\|_{L^4}\leq \|\bar u_0^N\|_{L^4}$, the two integrals are bounded by constant multiples of $\|\bar u_0^N\|_{L^4}\sup_{0<s<t}\|w^N(s)\|_{L^2}^{1/4}\, t^{1/8}\big(\int_0^t\|\nb w^N\|_{L^2}^2\,ds\big)^{7/8}$ and $\|\bar u_0^N\|_{L^4}^2\, t^{1/2}\big(\int_0^t\|\nb w^N\|_{L^2}^2\,ds\big)^{1/2}$, respectively. That (5) implies this inequality for every $N>0$ is shown in the proof of Lemma \ref{lemma.energy} (with $V_0=\bar u_0^N$). Conversely, if the inequality holds, then, since $\td u=w^N-e^{t\Delta}\td u_0^N$, we have $\limsup_{t\to0^+}\|\td u(t)\|_{L^2}\leq 2\|\td u_0^N\|_{L^2}$, and the right-hand side tends to zero as $N\to\I$ by Lemma \ref{lemma.split}.
\item Definition \ref{def.weak} also makes sense when $p=3$, and every $L^{3,\I}$-weak solution in the sense of Definition \ref{def.weak3} is an $L^{3,\I}$-weak solution in the sense of Definition \ref{def.weak}: when $p=3$ the right-hand side of \eqref{ineq:energyIneq} is finite and vanishes as $t\to0^+$, so \eqref{ineq:energyIneq} implies (5). Thus the issue addressed by condition (5) does not arise in the critical case, and the results of Section \ref{sec.weak} which are stated for $2<p\leq 3$ apply, in particular, to the $L^{3,\I}$-weak solutions of \cite{BaSeSv}. Here, as usual for suitable weak solutions, the pressure in Definition \ref{def.weak3} is understood to belong to $L^{3/2}_\loc$, as required in condition (1) of Definition \ref{def.weak}.
\end{enumerate}
\end{revblock}

	 We are able to push through the major results from \cite{BaSeSv} for this definition as recorded in the following three theorems. 
	 
	 \begin{theorem}[Existence]\label{thrm.existence}

	 Assume $u_0\in L^{p,\I}$ for some $2<p<3$ and is divergence free. Then there exists an $L^{p,\I}$-weak solution with initial data $u_0$. 
	 \end{theorem}

	\begin{theorem}[\textit{A priori} bound]\label{thrm.aprioribound}
		Assume $u_0\in L^{p,\I}$ for some $2<p<3$ and is divergence free. Let $u$ be an $L^{p,\I}$-weak solution with data $u_0$. Then,  letting $\sigma(p)=\frac 1 2 \frac {p-2}{4-p}\in (0,1/2)$, 
		\[
		\| u-e^{t\Delta}u_0\|_{L^2}^2(t)+\int_0^t  \| \nabla(u-e^{t\Delta}u_0)\|_{L^2}^2\,dt \leq      C_{p }   \left[
		\|u_0\|^\frac{2p}{4-p}_ {L^{p,\infty}}  
		    t^{\sigma(p)}
		+  t^{\frac 1 2}
		\right].
		\]
		\rev{In particular, this quantifies the rate of convergence in condition (5) of Definition \ref{def.weak}.}
	\end{theorem}
	 
	 \begin{theorem}[Stability under weak-star convergence]\label{thrm.stability}
	 	Assume $\{u_{0}^{(k)}\}\subset L^{p,\I}$ for some $2<p<3$ is a sequence of divergence free vector fields which converges in the weak-star sense to some divergence free $u_0\in L^{p,\I}$. Let $\{u^{(k)}\}$ be a sequence of $L^{p,\I}$-weak solutions each for the initial data $u_0^{(k)}$. Then there exists a sub-sequence $u^{(k_j)}$  that converges in the sense of distributions to an $L^{p,\I}$-weak solution $u$ with data $u_0$.
	 \end{theorem}

	 \medskip\noindent \textbf{Discussion of Theorems \ref{thrm.existence}-\ref{thrm.stability}:}
	 \begin{enumerate}
	 	\item The main task in proving these theorems is to establish the \textit{a priori} bound in Theorem \ref{thrm.aprioribound}. Once this is done, the   proofs of Theorems \ref{thrm.existence} and \ref{thrm.stability} follow their counterparts in \cite{BaSeSv}.
	 		 \item
	 	In view of the existence of self-similar $L^{3,\I}$-weak solutions \cite{JS,Tsai-DSSI,BT1,BaSeSv,AB}, 
	 	the \textit{ a priori} bound \eqref{ineq.BSSdecay} is \textit{sharp}. It would be interesting to establish the sharpness of the endpoint case of the corresponding decay estimate for $L^{p,\I}$-weak solutions when $p<3$, but doing so seems nontrivial. In particular, even though  $L^{p,\I}$   includes homogeneous initial data for $2<p<3$, homogeneity cannot be used because the initial data space is not scaling invariant under the Navier-Stokes scaling.   It would be interesting   to understand the loss of homogeneity in an  $L^{p,\I}$-weak solution with homogeneous data.
	 	\item Stability under weak-star convergence is important to analyze because it is possibly related to the question of whether or not Leray weak solutions satisfy the global energy \textit{equality}. Some discussion of this is available in \cite[Remark 1.3.3 on p. 631]{BaSeSv}. The occurrence of the \textit{a priori} bound in Theorem \ref{thrm.aprioribound} limits the rate at which high frequency activity in the initial data can be transferred to large scales in short periods of time---a dynamic that could lead to the failure of the energy equality. Our contribution shows that the same restriction applies for Leray weak solutions with data in $L^{p,\I}$ where $2<p<3$ as in $L^{3,\I}$. However, the restriction degrades as $p$ decreases to $2$ indicating a path for the failure of the energy equality. Proving the bounds in Theorem \ref{thrm.aprioribound} are sharp by way of counterexamples (see previous comment) could possibly provide a way to construct Leray weak solutions that \textit{do} transfer activity from small to large scales at arbitrary rates.
	 	\item 
	 	A mild solution to \eqref{eq.ns} is one that satisfies the formula
	 	\[
	 	u(x,t)=e^{t\Delta}u_0 +B(u,u), 
	 	\]
	 	where $\mathbb P$ is the Leray projection operator and
	 	\[
	 	B(f,g):=- \int_0^te^{(t-s)\Delta} \mathbb P \nb \cdot \big( f\otimes g )\big)\,ds.
	 	\]
	 	 Because $L^{p,\I}$-weak solutions have spatial decay, the pressure can be solved for from $u$ using Riesz transforms in a standard way---see \cite[p.~109]{LR} or \cite{JiaSverakIll,BT7}. There is an equivalence between mild solutions and solutions whose pressure is given by Riesz transforms \cite{LR,BT7}. Consequently, $L^{p,\I}$-weak solutions are also mild. We will use this fact frequently throughout the paper. 
	 
	\item 
	Popkin generalizes the work of Albritton and Barker \cite{AB} from the full range of critical non-endpoint Besov spaces between $L^{3,\I}$ and $\dot B^{-1}_{\I,\I}$ to the full range of supercritical non-endpoint  Besov spaces containing $L^{p,\I}$ for $2<p<3$,  but does not include a version of the decay estimate \eqref{ineq.BSSdecay}. We presently share our guess as to why this is the case.  In the context of Albritton and Barker, to accommodate the full range of critical Besov spaces, it is necessary to replace $\|u-e^{t\Delta}u_0\|_{L^2}$ with $\|u-P_k\|_{L^2}$ where $P_k$ is a higher \rev{Picard} iterate\footnote{Recall  the definition of Picard iterates: Let $P_0 = P_0 (u_0) = e^{t\Delta}u_0$ and define the $k^{\text{th}}$ Picard iterate recursively to be $P_k = P_0 + B(P_{k-1},P_{k-1})$.} chosen based on the Besov space in view. This is allowable in \cite{AB} because $\|u-P_0\|_{L^2}$ and $\|u-P_k\|_{L^2}$ satisfy the same estimate precisely because $\| P_k-P_{k-1}\|_2$ does for all $k$. Intuitively, in the sub-critical setting, $\| P_k-P_{k-1}\|_2$ gets smaller as $k$ increases while in the critical setting $\| P_k-P_{k-1}\|_2$ doesn't get smaller but it also doesn't get larger. However, in the supercritical setting, $\| P_k-P_{k-1}\|_2$ does, in principle, get larger as $k$ increases. Since higher values of $k$ must be considered to accommodate  rougher Besov spaces, but higher values of $k$ lead to a degradation in the supercritical setting, there seems to be no hope to get the decay estimate in Popkin's setting without restricting attention to a sub-scale of spaces.  For this reason, the formulation of Popkin's results more closely resembles the earlier work of Calder\'on. Based on the findings in this paper, we conjecture that for each $p\in(2,3)$ there is a range of Besov spaces close to $L^{p,\I}$ for which a version of Theorem \ref{thrm.aprioribound} can be established. 
	
	\rev{In a different direction, Luo \cite{Luo}, establishes norm inflation for \eqref{eq.ns} with smooth, compactly supported data in the supercritical Besov spaces $\dot B^s_{q,r}$ near the critical line $s=-1+3/q$ for every $s\neq 0$. The norm inflation occurs in $\dot B^s_{q,r}$. This leaves open whether norm inflation occurs in the supercritical Lebesgue spaces $L^q$, $2<q<3$, which correspond to $s=0$. This is consistent with  our conclusion that in the Lorentz scale  we obtain the decay estimate, and with it the stability under weak-star convergence in Theorem \ref{thrm.stability}. Understanding the persistence or failure thereof of the decay estimate across   the full range of Besov spaces   in \cite{Popkin} as well as the decay estimate's relationship to Luo's norm inflation are interesting topics for future research.} 
	\item The \textit{a priori} bound in Theorem \ref{thrm.aprioribound} can be extended to other space-time norms, a fact which is useful in applications as the power on the right-hand side can be  increased, indicating a tighter confinement. The details of this appear in Proposition \ref{prop.LpinftyDecaySpacteTime} which is used in our first application.
	\item For $2<p<3$, $L^{p,\I}$-weak solutions \textit{eventually regularize} in that there exists a time $T$ so that $\sup_{T<t<\I}\|u\|_{L^\I}<\I$ for all $t>T$. This follows from a result in \cite{BT8}. Interestingly, this property is not shared by $L^{3,\I}$-weak solutions which presumably do not necessarily eventually regularize---this would be the case if, e.g., discretely self-similar solutions can possess singularities. This discrepancy reflects the fact that large scales are suppressed by the finiteness of the $L^{p,\I}$ norm to a greater degree than they are by the finiteness of the $L^{3,\I}$ norm. The fact that the large time behavior is better when $p<3$ is also apparent in Remark \ref{rem:long-time} which discusses an improvement to Theorem \ref{thrm.aprioribound} as $t\to \I$.
\end{enumerate}
	 
	 \subsection{Application 1: Time-asymptotic expansions and separation rates}
	  Classically, the Picard iterates converge to a solution to \eqref{eq.ns} whenever \eqref{eq.ns} can be viewed as a perturbation of the heat equation.  This is the case for large data in $L^p$ for $p>3$ or for small data in $L^{3,\I}$ where a Picard scheme is used to construct solutions \cite{FJR,Kato}.  This is not known, and likely fails,  for large $L^{3,\I}$ data, so we do not expect convergence of $P_k$ to $u$ when $u$ is an $L^{3,\I}$-weak solution. It is even less likely to be the case in the supercritical class of  $L^{p,\I}$-weak solutions. In the case of $L^{3,\I}$-weak solutions, it is shown in \cite{BP2} that   Picard iterates do capture some local asymptotics at $t=0$ of  $L^{3,\I}$-weak solutions. In particular, that paper contains the following theorem,  \cite[Theorem 1.3]{BP2}, which generalizes a finding for discretely self-similar solutions in \cite{BP1}.  
	 \begin{theorem}[Local asymptotic expansion---critical scenario]\label{thrm.main2BP2}
	 	Assume $u_0\in L^{3,\I}$ and is divergence free. Let $u$ be an $L^{3,\I}$-weak solution with data $u_0$. Fix $x_0\in \R^3$ and $p\in (3,\I]$. Assume further that $u_0|_B\in L^p(B)$ where $B=B_2(x_0)$. Then, there {exist} $\ga =\ga (p)\in(0,1)$ and $T=T(p,\|u_0\|_{L^{3,\I}},\|u_0\|_{L^p(B)})>0$ so that, for any 
	 	$\sigma \in (0,3/2)$, $t\in (0,T)$ and $k=0,1,\ldots, k_0$,
	 	\EQ{\label{ineq:thrm.main2}
	 		\| u-P_k\|_{L^\I (B_{1/4}(x_0))} (t)\lesssim_{p,u_0,\sigma,k} t^{a_k}, 
	 	}
	 	where $a_0 = \min\{\ga/2, 1/2-3/(2p)\}$,
	 	$a_{k+1} =\min\left\{\sigma, k( 1/2 -3/(2p))+a_0\right\}$ and 
	 	$k_0$ is the smallest natural number so that 
	 	\[
	 	k_0\bigg(\frac 1 2 -\frac 3 {2p}\bigg) +a_0 \geq \sigma.
	 	\]
	 	In particular, $a_{k_0}=\sigma$ and $a_k>a_{k-1}$ for $k=1,\ldots,k_0$. It follows that, for $(x,t)\in B_{1/4}(x_0) \times (0,T)$, and letting $a_{-1}= -3/(2p)$, we have
	 	\[
	 	u(x,t)= P_0 + \sum_{k=0}^{k_0-1}\mathcal O (t^{a_k}) + \mathcal O(t^{\sigma}) = \sum_{k=-1}^{k_0}\mathcal O (t^{a_k}),
	 	\]
	 	where the $\mathcal O (t^{a_k})$ terms are exactly solvable for $-1\leq k<k_0$. \rev{Here the dependence on $u_0$ in \eqref{ineq:thrm.main2} is via the quantities $\|u_0\|_{L^{3,\I}}$ and $\|u_0\|_{L^p(B)}$.} 
	 \end{theorem}

	 Short-time asympototic expansions have been examined by Brandolese for small self-similar flows \cite{Brandolese} and by Brandolese and Vigneron for both small (in which case the expansion holds for all times) and large (in which case the data is globally sub-critical and the expansion is up to a finite time) non-self-similar flows \cite{BV}. A follow-up paper by Bae and Brandolese considers the {forced} Navier-Stokes equations \cite{BaeBrandolese}. In \cite{KR}, Kukavica and Ries give an expansion in arbitrarily many terms assuming the solution is smooth. In all of the preceding papers, either the initial data is strong enough to generate smooth solutions (e.g.~it is in a sub-critical class or is small in a critical class) or the solution is assumed to be smooth. Additionally, the terms of the asymptotic expansions depend on $u$.  
	 
	 As discussed earlier, there is considerable evidence that non-unique solutions exist within the $L^{3,\I}$-weak solution class. 
	 Theorem \ref{thrm.main2BP2} has an application to the problem of quantifying the possible severity of non-uniqueness from a point-wise perspective---this compliments \eqref{ineq.BSSdecay} which does the same but from an $L^2$ perspective. 	 Some  philosophically relevant papers on the phenomenon of separation are \cite{DrivasEtAl,PDS,B1,VasseurEtAl,VasseurEtAl2,TBM}.
	 The key to this application is that the asymptotics are uniquely determined by $u_0$. An upper bound on the separation rate can therefore be derived from Theorem \ref{thrm.main2BP2} and the triangle inequality. This is stated in the following theorem which is taken from \cite{BP2}.
	 \begin{theorem}[Separation rate estimate---critical scenario]\label{thrm.mainBP2}
	 	Assume $u_0\in L^{3,\I}$ and is divergence free.  Fix $x_0\in \R^3$. Assume that $u_0|_B\in L^p(B)$ where $B= B_2(x_0)$ and $p\in (3,\I]$. Let $u_1$ and $u_2$ be  $L^{3,\I}$-weak solutions with data $u_0$. Then, there exists $T=T(p,u_0)>0$ so that, for every $\bar \sigma \in (0,3/2)$ and $t\in (0,T)$,
	 	\[
	 	\| u_1-u_2\|_{L^\I (B_{1/4}(x_0))} (t)\lesssim_{ p,\bar \sigma,u_0} t^{\bar \sigma}, 
	 	\] 
	 	where the dependence on $u_0$ is via the quantities $\|u_0\|_{L^p(B)}$ and $\|u_0\|_{L^{3,\I}}$.
	 \end{theorem}

	 The proof of Theorem \ref{thrm.main2BP2} in \cite{BP2} is based on re-writing $u-P_{k+1}$ as
	 \EQ{\label{eq:splitting}
	 	u-P_{k+1} 
	 	= B(u-P_{k},u-P_{k}) + B(P_k, u-P_{k})+B(u-P_{k},P_k).
	 } 
	 In sub-critical regimes, estimates for the bilinear operator $B$ come with a positive power of $t$. If $u-P_k$ also has an upper bound involving a positive power of $t$, then these powers stack leading to an improved estimate for $u-P_{k+1}$ compared to that for $u-P_k$---this bootstrapping is sometimes referred to as a ``self-improvement property'' and has been used elsewhere, e.g., in \cite{GIP,Brandolese,AB,BP1,BP2}. As our initial datum is only locally sub-critical, the self-improvement property only applies locally. In other words, if we write 
	 \[
	 B(P_k, u-P_{k}) = B(P_k, (u-P_{k})(1-\chi_{B}))+B(P_k, (u-P_{k})\chi_{B}),
	 \]
	 for an appropriate ball $B$, then the self-improvement argument applies to $B(P_k, (u-P_{k})\chi_{B})$. For the far-field contribution, $B(P_k, (u-P_{k})(1-\chi_{B}))$, the decay property of $L^{3,\I}$-weak solutions,
	 \EQN{\label{higher}
	 \|u-P_k\|_{L^2}\lesssim t^{1/4},
	}
	 is used. This holds for $L^{3,\I}$-weak solutions for all $k\in \N$ and therefore can be used to bound part of each term in the asymptotic expansion. In particular, at each iteration, the far-field part separates at most at the same rate due to the above estimate on $u-P_k$ and this rate ends up being the limiting rate $\sigma$ in Theorem \ref{thrm.main2BP2}. 
	 Note that the exponent on the right-hand side of the estimate for $u-P_k$ is independent of $k$ which is a manifestation of criticality. 
	 
	 Our goal is to adapt these results to the $L^{p,\I}$ case.  There is an obstacle to doing so, however, which is that, while 
	 \[
	 \| u -P_0\|_{L^2}\lesssim t^{\sigma/2},
	 \]
	 this is not the case for $u-P_1$. There would be a degradation in the exponent on the right-hand side and this would only get worse at higher $k$. This is a consequence of supercriticality and is the main obstruction to extending the Proof of Theorem \ref{thrm.main2BP2} to the supercritical setting. To overcome this setback, we find a new way to manage the far-field contribution by building our higher order terms off of the local part of the lower order terms. This will allow us to prove the following theorem.

	 \begin{theorem}[Local asymptotic expansion---supercritical scenario]\label{thrm.asympt} Fix  $2<p<3$.
	 	Assume $u_0\in L^{p,\I}(\R^3)$ is divergence free and, for some ball $B$, $u_0|_{B}\in L^\I(B)$. Let $\sigma=\sigma(p)$ be as in Theorem \ref{thrm.aprioribound} and fix $\delta\in (0,\sigma)$. Let $u$ be an $L^{p,\I}$-weak solution for data $u_0$. Then for any ball $B_\Om\Subset B$, there \rev{exist} $T=T(B,B_\Omega,\|u_0\|_{L^\I(B)} ,\|u_0\|_{L^{p,\I}(\R^3)} )$ \rev{and a function $P_\Omega$ on $\R^3\times(0,\I)$, determined by $u_0$, $B$ and $B_\Omega$ alone (it is given explicitly after the statement of the theorem and, in particular, does not depend on $u$), such that}  
	 	\[
	 	(u - P_\Omega)(x,t) = O(t^{1+\sigma-\delta}),
	 	\]
	 	for $x\in B_\Omega$ and $t\in (0,T)$ \rev{ where the suppressed constant depends on $B,B_\Omega,\|u_0\|_{L^\I(B)} ,\|u_0\|_{L^{p,\I}(\R^3)}$.} 
	 	
	 	\smallskip 
	 	\noindent\emph{(Separation rate estimate---supercritical scenario)} As a consequence, if $v$ is another $L^{p,\I}$-weak solution with data $u_0$, then 
	 	\[
	 	|u - v|(x,t) \lesssim_{u_0,B_\Omega,B} t^{1+\sigma-\delta},
	 	\]
	 	for $x\in B_\Omega$ and $t\in (0,T)$ \rev{ where the suppressed constant depends on $B,B_\Omega,\|u_0\|_{L^\I(B)} ,\|u_0\|_{L^{p,\I}(\R^3)}$.} 
	 \end{theorem}
	 

	 In the preceding theorem the leading order term $P_\Omega$ is the sum of  $ P_1+\td P_2 $ where $P_1$ is the first Picard iterate after $P_0=e^{t\Delta}u_0$ and
	 \[
	 \td P_2 :=   B(P_0,  B(P_0,P_0) \chi_1)+ B( B(P_0,P_0),P_0 \chi_1),
	 \]
	 Here, $\chi_1$ is a cut-off function so that   $B_\Omega \Subset   \supp \chi_1 \Subset B$. Plainly $P_\Omega$ is uniquely determined  by $u_0$ and so the separation rate estimate follows simply from the triangle inequality.  One can prove a weaker statement with $ u_0|_{B}\in L^q(B)$ for $q\in (3,\I)$ to match Theorem \ref{thrm.mainBP2}, but it leads to technicalities which obfuscate the main point. We discuss this further in Remark \ref{remark.generalCase}. 
	 
	 The application to separation rates is only interesting if non-unique solutions turn out to exist. While this is likely the case, it is valuable to find an application where there is no risk the result is vacuous. This is the case for the following application which quantifies how fast two solutions whose initial data agree \textit{locally} can separate locally as a result of far-field differences in the data. It was suggested to the first author by Radu Dascaliuc. 
	 \begin{corollary}\label{cor.differentData} Fix $p\in (2,3)$.
	 	Assume $u_0$ and $v_0$ are in $L^{p,\I}(\R^3)$, are divergence free and satisfy $u_0|_{B}=v_0|_B\in L^\I(B)$ for some ball $B$.  Let $u$ and $v$ be $\rev{L^{p,\I}}$-weak solutions with data $u_0$ and $v_0$ respectively. Then,  under the assumptions and notation of Theorem \ref{thrm.asympt},
	 	\[
	 	|u-v|(x,t)\lesssim_{u_0,v_0,B_\Om,B}t ,
	 	\]
	 	for $x\in B_\Omega$ and $t\in (0,T)$. Here $T$ has dependencies on both $u_0$ and $v_0$.
	 \end{corollary}
	 Using the triangle inequality  and Theorem \ref{thrm.asympt}, proving this will reduce to proving a similar statement for $P_\Om(u_0)-P_\Om(v_0)$ and therefore boils down to an analysis of the heat equation. We will see that the non-local property of the Oseen tensor means that, when comparing, e.g.,~$P_1(u_0)$ and $P_1(v_0)$, the fact that the far-field does not vanish at $t=0$ means that the best one can hope for is a single power of $t$ on the right-hand side. This explains why the unitary power of $t$ is necessary here compared to the power in Theorem \ref{thrm.asympt}.  
	 
	 \subsection{Application 2: Time regularity at a singular time}

	 As remarked by Serrin as far back as the 1960s,   boundedness in $Q=B\times (0,T)$ implies infinite differentiability in $x$ within $Q$ but the pressure ruins things for differentiability in $t$ \cite{Serrin}. It is, however, easy to see that,
	 \[
	 \partial_t u \in L^\I((\epsilon,T]\times B'), 
	 \]
	 from spatial regularity, at least when $u$ is a weak solution in an appropriate class and $B'$ is compactly contained in $B$---see Lemma \ref{lemma.firstorder}. On the other hand, when a mild solution is globally bounded, it is known to become time-analytic \cite{DongZhang} at positive times.
	 In this application we seek orders of time regularity   between these extremes.
	  To do this we must bypass classes which imply boundedness. For example, replacing $L^\I(0,T; L^\I(\R^3))$ with $L^\I(0,T; L^p(\R^3))$ for $p>3$ does not change  anything as any solution in the latter will necessarily be in the former on $\R^3\times (\delta,T)$. Consequently, time analyticity will only break down  in the vicinity of a critical or supercritical  order
	 singularity.   This will be the setting of our result.

	For convenience we switch our perspective to solutions which live on a time interval containing zero on the interior and which have an isolated   singularity at the space-time origin. The most restrictive singularity  is a Type 1  singularity for which a solution satisfies the upper bound
	 \[
	 |u(x,t)|\lesssim 	\frac 1 {|x|+\sqrt{|t|}}.
	 \]
	 Any tighter bound would imply the solution is actually bounded, and hence the solution is time-analytic at the origin. 
	 We presently investigate this scenario (and supercritical versions of it) with the aim of showing that, for every $x\neq 0$, we have   $\partial_t u (x,\cdot) \in C^{0,\gamma}(I)$ for some exponent $\gamma=\gamma(p)\in (0,1)$, which arises from the $L^{p,\I}$-weak solution theory, and an interval $I$ which contains the singular time $t=0$. This amounts to  a higher-order time regularity result at a singular time and at spatial points away from the singularity. We are not aware of any comparable   time-regularity results in the Navier-Stokes literature.

	 \begin{theorem}[Time regularity at a singular time]\label{thrm.time-reg}
	 	Let $p\in (2,3]$ be fixed.
	 	Suppose that $u(x,t)$ is an $L^{p,\I}$-weak solution to \eqref{eq.ns} on $\R^3\times (-\delta, \delta)$ \rev{(in the sense of Definition \ref{def.weakinterval})} and satisfies
	 	\EQ{\label{ineq.assumption}	
	 	|u(x,t)|\leq \frac {M} {(|x| +\sqrt{|t|})^{3/p}},
	 }
	 	for a constant $M>0$ and a fixed $\delta>0$.
	 	So, there may be  an isolated singularity belonging to $L^{p,\I}$ at $(x,t)=(0,0)$. 
	  Let $\sigma(p)=\frac 1 2 \frac {p-2}{4-p}\in (0,1/2]$.
	 	Additionally assume:
 			\[ \sup_{ t\in (-\delta, \delta)} \|u(\cdot, t)\|_{\dot H^{\frac 3 2 - \frac 3 p}} < M \text{ if } 2<p<3 \text{ or  }  \sup_{ t\in (-\delta, \delta)} \|(-\Delta)^{\frac 1 2 (\frac 3 2 - \frac 3 p)}u(\cdot, t)\|_{L^{2,\I}} < M \text{ if }2<p\leq 3.\]
	 	Then, for $x\neq 0$, we have 
	 	\[
	 	\| u(x,\cdot)\|_{C_t^{1,\sigma/2} (-\delta/4,\delta/4)} \leq C(M,p,|x|\rev{,\delta}),
	 	\]
	 	if $2<p<3$ while 
	 		\[
	 	\| u(x,\cdot)\|_{C_t^{1,(\sigma(p)-)/2} (-\delta/4,\delta/4)} \leq C(M,p,|x|,\sigma(p)-\rev{,\delta}),
	 	\]
	 	if $p=3$.
	 \end{theorem}

	 The notation $\sigma(p)-$ in the theorem means the result holds for any value strictly less than $\sigma(p)$.
	 The additional assumptions, e.g., about the $\dot H^{3/2-3/p}$-norm, have the same scaling as \eqref{ineq.assumption}. The Sobolev space assumption is only meaningful  in the supercritical regime as, in the critical regime, it would imply regularity so there would be no singularity. This is because $ L^\I_t\dot H^{1/2}_x\subset L^\I_t L^3_x $ indicating it is a regularity class \cite{ESS}.
	 This explains why, when $p=3$, we have included the weaker assumption involving the $L^{2,\I}$ norm. In particular, it  should not   rule out the singularity unless an additional smallness condition is imposed. It turns out that the weaker assumption doesn't impact the final conclusion if $2<p<3$, so the assumption about the $\dot H^{3/2-3/p}$ norm is superfluous---it is included for illustrative purposes  as it is perhaps less exotic than the Lorentz space quantity. 
 
	 The proof reflects the fact that the pressure is responsible for the limited time regularity due to its \textit{nonlocality}. In particular, it is the far-field part of the pressure which can take far-field spatial irregularity and disturb the local time regularity. The   proof applies the $L^{p,\I}$-weak solution theory to control the far-field contribution of the pressure, which leads to the appearance of $\sigma(p)$.  This quantifies the extent to which   far-field effects can disrupt  time regularity locally.

	 \subsection*{Organization} Section \ref{sec.weak} contains the weak solution theory including the proofs of Theorems \ref{thrm.existence}-\ref{thrm.stability}. Section \ref{sec.asy} deals with the asymptotic expansion and Section \ref{sec.reg} deals with time regularity.

         \section{Weak  solutions}\label{sec.weak}
         
         In this section we develop the theory of   $L^{p,\I}$-weak solutions following the approach taken for $L^{3,\I}$-weak solution in \cite{BaSeSv}.
         The endpoint Lorentz spaces $L^{p,\I}$ coincide with the weak-$L^p$ spaces and we forgo a definition of the former. Note that in this paper we do not deal with these spaces at a technical level. The main facts that we need about Lorentz spaces are recorded here.  We make frequent use of the following inequalities where $1\leq p,q \leq \I$,
           \begin{equation}\label{eqn:Young-O'Neil-type}
               \| f * g \|_{L^\infty} \leq \| f \|_{\rev{L^{q,1}}} \| g \|_{L^{p,\infty}}, \quad \frac1p + \frac1q = 1,
           \end{equation}
           and
           \begin{equation}\label{eqn:heat-persistance}
               \| e^{t\Delta} u_0 \|_{L^{p,\infty}} \leq \|u_0\|_{L^{p,\infty}}.
           \end{equation}
           The first is due to O'Neil \cite{ONeil}, see also \cite{DDN}, and the second is a consequence of it which can be found in \cite{BCD}.
          We also have, \EQ{\label{ineq.aaa}
               \| e^{t\Delta} u_0 \|_{L^p} \leq (4\pi t)^{-\tfrac32\left(\tfrac1q - \tfrac1p\right)} \| u_0 \|_{L^{q,\I}}, \text{ }(t>0).
           }
           
           The most important property of Lorentz spaces in our analysis is the following  splitting lemma due to  Barker, Seregin and \SVERAK ~ \cite{BaSeSv}, which follows in the spirit of the work of C.~Calder\'on \cite{Calderon}. See also generalizations to Besov spaces in \cite{AB,Popkin}.
           \begin{lemma}[{Lorentz space decomposition of Barker, Seregin and \SVERAK}]\label{lemma.split}
           
           If \(1 < t < r < s \leq \infty\), then 
           \(L^{r,\infty} \subset L^s + L^t\). In particular, if
           \(u_0 \in L^{r,\infty}\)  and \(\nabla \cdot u_0 = 0\) in
           \(\mathcal{D}'\), then for any \(N > 0\), there exists a decomposition
           \(u_0 = \bar{u}^N_0 + \tilde{u}^N_0\), where \(\bar{u}^N_0\) and
           \(\tilde{u}^N_0\) are divergence free, and \[ 
               \tilde{u}^N_0 \in [C_{c,\sigma}^\infty(\mathbb{R}^3)]^{L^t(\mathbb{R}^3)}, 
               \quad \| \tilde{u}^N_0 \|_ {L^t}^t \leq \frac{C_r}{r-t} \frac{1}{N^{r-t}} \|u_0\|^r_{L^{r,\infty}},
           \] and if \(s < \infty\), \[ 
               \bar{u}^N_0 \in [C_{c,\sigma}^\infty(\mathbb{R}^3)]^{L^s(\mathbb{R}^3)}, 
               \quad \| \bar{u}^N_0 \|_ {L^s}^s \leq \frac{C_s}{s-r} N^{s-r} \|u_0\|^r_{L^{r,\infty}},
           \] otherwise if \(s = \infty\), \[ 
               \bar{u}^N_0 \in [C_{c,\sigma}^\infty(\mathbb{R}^3)]^{L^\infty(\mathbb{R}^3)}, 
               \quad \| \bar{u}^N_0 \|_ {L^\infty} \leq C_s N.
           \]
           Above, $C_{c,\sigma}^\infty(\mathbb{R}^3)$ denotes smooth, divergence free compactly supported vector functions.
           \end{lemma}
           For the purposes of dimensional analysis, it is helpful to observe that $N$ has the same dimension as $u_0$.

           \subsection{\textit{A priori} bounds}

           In this subsection we prove Theorem \ref{thrm.aprioribound}. 
           In the following lemma, we quantify the growth of separation of a solution of  \eqref{eq.ns} and the heat equation, in terms of their initial separation and a term appearing as a source due to the difference between the equations. The amount of separation is quantified as the total energy of the difference.
           \begin{lemma}[Energy estimate for subcritical perturbations of $L^{p,\I}$-weak solutions]\label{lemma.energy}
           Let \(u\) be an $L^{p,\I}$-weak solution on \([0,T]\) with divergence free data $u_0\in L^{p,\I}$. 
           Let \(w = u - V\), where \(V\) is the caloric extension of \(V_0\), and 
           \(V_0 \in \rev{L^\alpha}\) \rev{for some $\alpha\in(3,4]$} and is divergence free. Assume further that $ w_0 = \rev{u_0-V_0}\in L^2$. Then, for any \(t\in[0,T]\),
           \[ 
               \|w(t)\|_{L^2}^2 +\frac 1 2 \int_0^t   \|\nabla w(\tau)\|_{L^2}^2 d\tau \leq \mu(0,t) \|w_0\|_{L^2}^2 +\mu(0,t) \int_0^t \e_2^{-1} \| V(\tau) \|^4_{L^4}\,d\tau,
           \] 
           where  
            \[
           \mu(s,t) = e^{\frac{C_L^8}{4} \int_s^t \epsilon_1^{-7} \|V\|_{L^4}^8\,d\tau},
           \]
           for  $\epsilon_1 = 3/7$ and $\epsilon_2 = 3/4$.
           \end{lemma}

           \begin{proof}Observe that $u-e^{t\Delta}u_0 \in L^{\I}(0,T;L^2)\cap L^2(0,T;\dot H^1)$ by assumption. Then, $u-V = u-e^{t\Delta}u_0 + e^{t\Delta}u_0 - V$, implying by our assumptions that $u-V \in L^{\I}(0,T;L^2)\cap L^2(0,T;\dot H^1)$ as well. \rev{Moreover, $w=\td u+e^{t\Delta}w_0$ with $w_0\in L^2$, so condition (5) of Definition \ref{def.weak} and the strong continuity of the heat semigroup on $L^2$ imply that $w(s)\to w_0$ in $L^2$ as $s\to 0^+$. Finally, by \eqref{ineq.aaa}, $\|V(\tau)\|_{L^4}\leq (4\pi\tau)^{-\frac 3 2(\frac 1\alpha-\frac 1 4)}\|V_0\|_{L^\alpha}$, so $\|V\|_{L^4}\in L^8(0,T)$ because $\alpha>3$.}
           From the equations for \(u\) and \(V\), we have
           \(\nabla \cdot w = \nabla \cdot u - \nabla \cdot V = 0\), and
           \[ \partial_t w - \Delta w + (u \cdot \nabla) u + \nabla \rev{\pi} = 0. \]
           We know that $u$ is a suitable weak solution and we know that $V$ satisfies an energy \textit{equality} and a local energy equality. It is possible to deduce from this that $w$ satisfies a local energy inequality. Because $V$ is very regular, this is an easy calculation  and is omitted.

           Formally taking the inner-product of \(w\) with the evolution equation for \(w\) yields,
           \[ \frac12\frac{d}{dt} \|w\|_{L^2}^2 + \|\nabla w\|_{L^2}^2 + \langle (u\cdot\nabla) u, w \rangle = 0. \]
           From the cancellation property of the trilinear form we have
           \begin{align*}
               \langle (u\cdot\nabla) u, w\rangle 
               & 
               = -\langle (w \cdot \nabla)\ w, V\rangle \rev{-\langle (V \cdot \nabla)\ w, V \rangle}.
           \end{align*}
           These identities can be integrated in time to yield an energy equality.
           Note that these are obtained formally. However,  they can be made rigorous with the substitution of an inequality for the equality in the ensuing energy inequality by exploiting the local energy inequality on a sequence of test functions which approach $1$ on all of $\R^3$---the details of this argument are worked out in \cite[Lemma 3.3]{BaSeSv}.  To this point, because \rev{$\|V\|_{L^4}\in L^8(0,T)$},  \cite[Proof of Lemma 3.3]{BaSeSv} applies here  with no modifications except  that the $L^{3,\I}$ norms on the right-hand side of \cite[(3.28) \and (3.29)]{BaSeSv} are at this stage left as $L^4$ norms \rev{of $V$}. \rev{This argument uses test functions that vanish near the left endpoint of the time interval under consideration, thereby producing the following energy inequality on intervals $[s,t]$ for every $s\in(0,T)$ which is a Lebesgue point of the $L^2$-valued map $\tau\mapsto\td u(\tau)$, in particular for a.e.~$s>0$ (see \cite[(3.26)]{BaSeSv} for the genesis of this),
\EQ{\label{ineq.energy.s}
\|w(t)\|_{L^2}^2 +2\int_s^t \|\nabla w\|_{L^2}^2\,d\tau \leq \|w(s)\|_{L^2}^2 + 2\int_s^t\int ( V\otimes w +V\otimes V  ): \nabla w \,dx\,d\tau.
}
We now let $s\to 0^+$ in \eqref{ineq.energy.s} along such points. The integrals on the right-hand side converge to the corresponding integrals over $(0,t)$ because $\|V\|_{L^4}\in L^8(0,T)$ and $w\in L^\I(0,T;L^2)\cap L^2(0,T;\dot H^1)$, while $\|w(s)\|_{L^2}\to \|w_0\|_{L^2}$ by the observation made at the beginning of the proof, which uses  condition (5) of Definition \ref{def.weak}.} In particular, \rev{writing $w(0)=w_0$,} this confirms
           \[
           	\|w(t)\|_{L^2}^2 +2\int_0^t \|\nabla w\|_{L^2}^2 \leq \|w(0)\|_{L^2}^2 + 2\int_0^t\int ( V\otimes w +V\otimes V  ): \nabla w \,dx\,ds.
           \]
           From H\"older's inequality we obtain
           \[
           	 		\|w(t)\|_{L^2}^2 +2\int_0^t \|\nabla w\|_{L^2}^2 \leq \|w(0)\|_{L^2}^2+  2 \int_0^t ( \| w \|_ {L^4} \|\nabla w\|_{L^2} \|V\|_{L^4} + \|V\|_{L^4}^2 \|\nabla w\|_{L^2})\,ds.
          \]
          By 
           Ladyzhenskaya's inequality and Young's inequality we have 
           \begin{align*}
         \| w \|_ {L^4} \|\nabla w\|_{L^2} \|V\|_{L^4} + \|V\|_{L^4}^2 \|\nabla w\|_{L^2}
         &\leq C_L \|w\|_{L^2}^\frac{1}{4} \|\nabla w\|_{L^2}^\frac{7}{4} \|V\|_ {L^4} + \|V\|_{L^4}^2 \|\nabla w\|_{L^2}
         \\&\leq 	\left( \tfrac{7}{8} \epsilon_1 + \tfrac{1}{2} \epsilon_2 \right) \|\nabla w\|_{L^2}^2 
         + \tfrac{C_L^8}{8}\tfrac{1}{\epsilon_1^7} \|w\|_{L^2}^2 \|V\|_{L^4}^8
         + \tfrac1{2\epsilon_2}\|V\|_{L^4}^4.
           \end{align*}
           For any $\delta \in [0,1)$, we can choose $\epsilon_1, \epsilon_2 > 0$ such that
           \[
           	\tfrac78 \epsilon_1 + \tfrac12 \epsilon_2 + \tfrac12 \delta = 1,
           \]
           and therefore,
           \[
           	\|w(t)\|_{L^2}^2 \leq  \|w(0)\|_{L^2}^2+\int_0^t \frac1{\epsilon_2}\|V\|_{L^4}^4  \,ds-\delta \int_0^t \|\nabla w\|_{L^2}^2\,ds +
           		\int_0^t 	\frac{C_L^8}{4}\tfrac{1}{\epsilon_1^7} \|w\|_{L^2}^2 \|V\|_{L^4}^8\,ds.
           \]
           We obtain from Gr\"onwall's inequality that 
              \begin{align*}
           \|w(t)\|_{L^2}^2  &\leq  \|w(0)\|_{L^2}^2+\int_0^t \frac1{\epsilon_2}\|V\|_{L^4}^4  \,ds-\delta \int_0^t \|\nabla w\|_{L^2}^2 \,ds
           \\&+\int_0^t\bigg(		\|w(0)\|_{L^2}^2+\int_0^s \frac1{\epsilon_2}\|V\|_{L^4}^4  \,dr-\delta \int_0^s \|\nabla w\|_{L^2}^2\,dr	\bigg)	\frac{C_L^8}{4}\tfrac{1}{\epsilon_1^7} \|V\|_{L^4}^8\mu(s,t)\,ds,
           \end{align*}
           where  
           \[
           	\mu(s,t) = e^{\frac{C_L^8}{4} \int_s^t \epsilon_1^{-7} \|V\|_{L^4}^8\,d\tau}.
           \]
           Re-arranging and dropping a negative term from the right-hand side results in 
                \begin{align*}
           	\|w(t)\|_{L^2}^2+ \delta \int_0^t \|\nabla w\|_{L^2}^2 \,ds &\leq  \|w(0)\|_{L^2}^2+\int_0^t \frac1{\epsilon_2}\|V\|_{L^4}^4  \,ds
           	\\&+\int_0^t\bigg(		\|w(0)\|_{L^2}^2+\int_0^s \frac1{\epsilon_2}\|V\|_{L^4}^4  \,dr	\bigg)	\frac{C_L^8}{4}\tfrac{1}{\epsilon_1^7} \|V\|_{L^4}^8\mu(s,t)\,ds.
           \end{align*}
           A standard calculation on the right-hand side using the fact that 
           \[
            \|w(0)\|_{L^2}^2+\int_0^t \frac1{\epsilon_2}\|V\|_{L^4}^4  \,ds,
           \]
           is increasing and the fundamental theorem of calculus let's us bound the right-hand side above by
           \[
           \mu(0,t) \|w(0)\|_{L^2}^2 + \int_0^t \epsilon_2^{-1} \mu(0,t) \| V(\tau) \|^4_{L^4}\,d\tau,
           \]
           which leads to our final estimate
           \[ 
               \|w(t)\|_{L^2}^2 + \int_0^t \delta  \|\nabla w(\tau)\|_{L^2}^2 d\tau \leq \mu(0,t) \|w_0\|_{L^2}^2 + \int_0^t \epsilon_2^{-1} \mu(0,t) \| V(\tau) \|^4_{L^4}\,d\tau.
           \] 
           Choosing $\delta = 1/2$, $\epsilon_1 = 3/7$ and  $\epsilon_2 = 3/4$ completes the proof.

           \end{proof}
           
           The previous theorem is framed for a solution of \eqref{eq.ns} and a solution of the heat equation with mostly unrelated initial conditions. In the next result, following the methodology of \cite{BaSeSv}, we examine the separation of the solutions with identical initial conditions. We do so by first separating the initial data into a supercritical and a subcritical part. Then, we apply the previous result to quantify the separation rate for the solution of \eqref{eq.ns} with the full initial condition and the solution of the heat equation with the subcritical initial condition (the separation therefore evolves from the supercritical part of the initial data).
           
           \begin{theorem}\label{thm:splitting}
           Let \(2 < p \leq 3,\) and suppose \(u_0 \in L^{p,\infty}\) is divergence free. Let  \(u\)   satisfy Definition \ref{def.weak} on \([0,T]\). 
           Let \(V\) be the caloric extension of \(u_0\), and let \(w = u - V\). For any \(\alpha \in (3,4],\) there exists a constant, \(C_{\alpha,p}\), such that for all \(t \in [0,T]\) and for any \(N > 0\),
           \begin{align*}
              & \|w(t)\|_{L^2}^2 +\frac 1 2 \int_0^t \|\nabla w\|_{L^2}^2\,ds 
               \\&\leq 
           	C_{\alpha,p}\exp\left( C_{\alpha,p} \|u_0\|_ {L^{p,\infty}}^{\frac{8 p}{\alpha}} N^{\frac{8}{\alpha}(\alpha - p)} t^{\frac{4\alpha - 12}{\alpha}} \right)
           	    \left[
           		\|u_0\|^{p}_ {L^{p,\infty}}\frac{1}{N^{ {p-2}}}
           		+ 
           		\|u_0\|_ {L^{p,\infty}}^{\frac{4p}\alpha} N^{\frac{4}{\alpha}(\alpha-p)} t^{\frac{5\alpha - 12}{2\alpha}}
           	    \right].
           \end{align*}
           \end{theorem}
           
           \begin{proof}
           Decompose \(u_0\) as in the Lorentz space decomposition lemma,
           \(u_0 = \bar{u}^N_ 0 + \tilde{u}^N_ 0 \in L^\alpha + L^2 \),
           where \[ 
               \| \bar{u}^N_ 0 \|_ {L^\alpha}^\alpha \leq \frac{C_\alpha}{\alpha-p} N^{\alpha-p} \|u_0\|^p_{L^{p,\infty}}
               , \quad
               \| \tilde{u}^N_ 0 \|_ {L^2}^2 \leq \frac{C_p}{p-2} \frac{1}{N^{p-2}} \|u_0\|^p_{L^{p,\infty}}.
           \]
           Let \(V^N\) denote the caloric extension of \(\bar{u}^N_ 0\) (the sub-critical part of the initial condition), and let \(w^N\) be the perturbation: 
           \[ w^N = u - V^N. \] 
           Then \(w^N(0) = u_0 - \bar{u}^N_0 = \tilde{u}^N_0\), 
           and from the previous lemma \rev{ (applied with $V_0=\bar u_0^N\in L^\alpha$, $\alpha\in(3,4]$)}, we have  
           \begin{multline*}
               \|w^N(t)\|_{L^2}^2 +\frac 1 2 \int_0^t \|\nabla w^N\|_{L^2}^2\,ds \leq \mu(0,t) \|\tilde{u}^N_ 0\|_{L^2}^2 +\mu(0,t) \int_0^t \epsilon_2^{-1} \| V^N(\tau)\|^4_4\,d\tau
               \\
               \leq \mu(0,t) \frac{C_p}{p-2} \frac{1}{N^{p-2}} \|u_0\|_ {L^{p,\infty}}^p +\mu(0,t) \int_0^t \epsilon_2^{-1}  \| V^N(\tau)\|^4_4\,d\tau.
           \end{multline*}
       From \eqref{ineq.aaa}, we have 
           \[ 
               \| V^N(\tau) \|_ 4 
           	\leq (4\pi \tau)^{-\frac32(\frac1\alpha - \frac14)} \| \bar{u}^N_ 0\|_ \alpha  
           	\leq (4\pi \tau)^{-\frac32(\frac1\alpha - \frac14)} 
           		\left(\frac{C_\alpha}{\alpha-p} N^{\alpha-p} \|u_0\|^p_{L^{p,\infty}}\right)^{\frac{1}{\alpha}}.
           \]
           Therefore, 
           \[
               \mu(0,t) \leq \bar{\mu}(0,t) := \exp\bigg(C_{\alpha,p}
           		N^{\frac{8}{\alpha}(\alpha-p)} \|u_0\|^\frac{8p}{\alpha}_{L^{p,\infty}}
           	\int_0^t \epsilon_1^{-7} \tau^{-\frac{12}\alpha + 3}\,d\tau\bigg),
           \]
           and
           \begin{align*}
             &  \|w^N(t)\|_{L^2}^2 +\frac 1 2 \int_0^t \|\nabla w^N\|_{L^2}^2\,ds
           \\&	\leq \bar{\mu}(0,t) \tfrac{C_p}{p-2} \frac{1}{N^{p-2}} \| u_ 0 \|_ {L^{p,\infty}}^p 
           	+ C_{\alpha,p} N^{\frac{4}{\alpha}(\alpha-p)} \mu(0,t)\|u_0\|^{\frac{4p}{\alpha}}_{L^{p,\infty}}
           	\int_0^t \epsilon_2^{-1} \tau^{-6(\frac1\alpha - \frac14)} \,d\tau
           	\\&
           	\leq \bar{\mu}(0,t) \tfrac{C_p}{p-2} \frac{1}{N^{p-2}} \| u_ 0 \|_ {L^{p,\infty}}^p 
           	+ \bar{\mu}(0,t) C_{\alpha,p} N^{\frac{4}{\alpha}(\alpha-p)} \|u_0\|^{\frac{4p}{\alpha}}_{L^{p,\infty}} 
           	\int_0^t \epsilon_2^{-1} \tau^{-6(\frac1\alpha - \frac14)} \,d\tau.
           \end{align*}
The time integrals converge because \( \alpha \in (3,4] \). In particular,
           \[
               \int_0^t \epsilon_1^{-7} \tau^{-\frac{12}\alpha + 3}\,d\tau
               = \epsilon_1^{-7} \tfrac{\alpha}{4\alpha - 12} t^{\frac{4\alpha - 12}{\alpha}}
           \quad \text{and}\quad
               \int_0^t \epsilon_2^{-1} \tau^{-6(\frac1\alpha - \frac14)} \,d\tau
               = \epsilon_2^{-1} \tfrac{2\alpha}{5\alpha - 12} t^{\frac{5\alpha - 12}{2\alpha}}.
           \]
           Therefore,  
           \[
           	\bar{\mu}(0,t) = \exp\big(\epsilon_1^{-7} C_{\alpha,p}
           		N^{\frac{8}{\alpha}(\alpha-p)} \|u_0\|^\frac{8p}{\alpha}_{L^{p,\infty}}
           		t^{\frac{4\alpha - 12}{\alpha}}
           	\big),
           \]
           and
           \[
               \|w^N(t)\|_{L^2}^2 +\frac 1 2 \int_0^t \|\nabla w^N\|_{L^2}^2\,ds
           	\leq C_{\alpha,p}\, \bar{\mu}(0,t) \left( 
           		\| u_ 0 \|_ {L^{p,\infty}}^p \frac{1}{N^{p-2}} 
           		+ \epsilon_2^{-1} \|u_0\|^{\frac{4p}{\alpha}}_{L^{p,\infty}}  N^{\frac{4}{\alpha}(\alpha-p)} t^{\frac{5\alpha - 12}{2\alpha}}
           	\right),
           \]
           where we redefine \(C_{\alpha,p}\) appropriately, and note that \(C_{\alpha,p}\to\infty\) as \(\alpha \to 3\) and \(C_{\alpha,p} \to \infty\) as \( p \to 2 \).
           
           Now, let \(W^N\) be the caloric extension of \(\tilde{u}^N_0\). By the
           linearity of the heat equation, \[ w = u - V = u - V^N - W^N = w^N - W^N. \]
           Therefore, 
           \EQN{ \|w(t)\|_{L^2}^2 +\frac 1 2 \int_0^t \|\nabla w\|_{L^2}^2\,ds&\leq2 \|w^N(t)\|_{L^2}^2 +2 \int_0^t \|\nabla w^N\|_{L^2}^2\,ds
           	\\&\quad + 
           2\|W^N(t)\|_{L^2}^2+2 \int_0^t \|\nabla W^N\|_{L^2}^2\,ds. }
           {By energy methods for the heat equation}  and the Lorentz space decomposition,
           \[ 
               \|W^N(t)\|_{L^2}^2 +\frac 1 2 \int_0^t \|\nabla W^N\|_{L^2}^2\,ds
           	\leq \|\tilde{u}^N_0\|_{L^2}^2
           	\leq \left(\frac{C_p}{p-2}\right)  \frac{1}{N^{ {p-2}}} \|u_0\|^p_ {L^{p,\infty}}. 
           \]
           Therefore,   \[
               \|w(t)\|_{L^2}^2+\frac 1 2 \int_0^t \|\nabla w\|_{L^2}^2\,ds \leq 
           	C_{\alpha,p}\, \bar{\mu}(0,t) 
           	    \left[
           		\|u_0\|^p_ {L^{p,\infty}} \frac{1}{N^{{p-2}}}
           		+ \epsilon_2^{-\frac12}\|u_0\|_ {L^{p,\infty}}^{\frac {4p}\alpha} N^{\frac{4}{\alpha}(\alpha-p)} t^{\frac{5\alpha - 12}{2\alpha}}
           	    \right].
           \]
           where we again redefine \(C_{\alpha,p}\) appropriately.  
           \end{proof}

           \begin{corollary}\label{cor:apriori-bound}
           Under the same conditions as in Theorem~\ref{thm:splitting},
          we have that, for all $t>0$,
          \[
          \|w(t)\|_{L^2}^2 + \frac12 \int_0^t \|\nabla w\|_{L^2}^2 \,ds\leq C_{p}   \left[
          \|u_0\|^\frac{2p}{4-p}_ {L^{p,\infty}}  t^{\frac 1 2 \frac {p-2}{4-p}}
          +  t^{\frac 1 2}
          \right].
          \]
           \end{corollary}

           This estimate is dimensionally balanced. Note that for short periods of time the first power of $t$ dominates while, for large times, the second power of $t$ dominates.  Also, this is essentially a restatement of Theorem \ref{thrm.aprioribound} so, in proving it, we will also be proving Theorem \ref{thrm.aprioribound}.  
          
          \begin{remark}
          	 \label{rem:long-time}
           Corollary~\ref{cor:apriori-bound} is written to emphasize the case $t\to0$. For $1\lesssim t$ bounded away from $0$, we can get an improved rate as $t\to\infty$, namely,
           $$ \|w(t)\|_{L^2}^2 + \frac12 \int_0^t \|\nabla w\|_{L^2}^2\,ds \leq C_{p} \|u_0\|_{L^{p,\infty}}^{p} t^{\frac{p-2}{2}},$$
           by choosing $N$ appropriately and $\alpha = 4$. We do not have an application in mind for this but believe it may be useful. In the limit $p\to 2^+$, it is consistent with the energy estimate for Leray weak solutions (although one would consider $L^2$ and not $L^{2,\I}$ in that case).    Intuitively, this improvement is because the low modes are suppressed by membership in  $L^{p,\I}$ and, as these are the modes that contribute to long-time activity, the long-time error is consequently weaker.
             \end{remark}
           
           \begin{proof}[Proof of Corollary \ref{cor:apriori-bound} and Theorem \ref{thrm.aprioribound}.]
           	We take 
           	\[
           	N = t^{\frac {12-4\alpha} {8(\alpha-p)}} \|u_0\|_{L^{p,\I}}^{\beta},
           	\]
           	where 
           	\[
           	\beta = \frac {p}{3-p}\bigg(-1 - \frac {3-\alpha} {\alpha-p} \bigg) = \frac p {p-\alpha},
           	\]
           	is introduced to keep the dimensions correct. To clarify this, note that $N$ has the same dimension as $w$ and $u_0$. Then, taking the dimension of $t$ to be the length scale   squared, we get that the right-hand side is a length scale to the power
           	\[
           	\frac {12-4\alpha} {4(\alpha-p)} + \frac p {p-\alpha} \frac {3-p}p = \frac {\alpha-p} {p-\alpha},
           	\]
           	which matches the left-hand side.
           	This choice ensures that $\bar \mu(0,t)$ is independent of $N$, $t$ and $\|u_0\|_{L^{p,\I}}$. With this choice and taking $\alpha=4$, we have that
           	\[
           	N^{-(p-2)} = t^{\frac 1 2 \frac {p-2}{4-p}} \|u_0\|_{L^{p,\I}}^{-p\frac {p-2}{p-4}},
           	\]
           	and 
           	\[
           	N^{\frac 4 \alpha (\alpha-p)}t^{\frac{5\alpha-12}{2\alpha}} = t^{\frac 1 2} \|u_0\|_{L^{p,\I}}^{-p }.
           	\]
           	The corollary follows.

           \end{proof}

           \subsection{Stability and existence} 
           In this section we prove Theorems \ref{thrm.stability} and \ref{thrm.existence}.
           \begin{proof}[Proof of Theorem \ref{thrm.stability}]     Let $u_0^{(k)}$ be given as in Theorem \ref{thrm.stability} with weak-star limit $u_0$. Let $M = \sup_{k\in \N} \|u_0^{(k)}\|_{L^{p,\I}}$\rev{, which is finite by the uniform boundedness principle}. Let $V^{(k)}$ and  $V$ be the caloric extensions of $u_0^{(k)}$ and $u_0$.  It is straightforward to see that $V^{(k)}\to V$ in the sense of distributions---see \cite[Proposition 2.5]{BaSeSv}. \rev{In fact, more is true. For $t>0$ the heat kernel $\Phi(x-\cdot,t)$ belongs to $L^{p',1}$, the predual of $L^{p,\I}$, so $V^{(k)}(x,t)\to V(x,t)$ for every $(x,t)\in\R^3\times(0,\I)$. Moreover, by \eqref{eqn:Young-O'Neil-type}, for each $n\in\N$ the functions $V^{(k)}$ and their first derivatives in $x$ and $t$ are bounded on $\R^3\times[n^{-1},\I)$ uniformly in $k$, so, by the Arzel\`a--Ascoli theorem, $V^{(k)}\to V$ uniformly on compact subsets of $\R^3\times[n^{-1},\I)$. Weak-star convergence of the data does not, however, give convergence of $V^{(k)}$ in a global norm, so all strong convergences below are local in space; this suffices because all test functions below have compact support. By the weak-star lower semicontinuity of the norm, $\|u_0\|_{L^{p,\I}}\lesssim_p M$.} Fix $T\in (0,\I)$. We presently  find a convergent subsequence  on $\R^3\times (0,T)$ and then can pass to further subsequences to get the result for larger $T$ and, by extension, for all of $\R^3\times (0,\I)$.

           Let $w^{(k)}= u^{(k)}-V^{(k)}$. Then, 
           \[
           		\|w^{(k)}\|_{L^2}^2(t)+\int_0^t \|\nabla w^{(k)}\|_{L^2}^2 \,ds \leq C_{p}   \left[
           	M^\frac{2p}{4-p}  \rev{t}^{\frac 1 2 \frac {p-2}{4-p}}
           		+  \rev{t}^{\frac 1 2}
           		\right]\rev{,\qquad t\in (0,T],}
           \]
            which provides a uniform bound on our sequence. 
            We additionally have for $B(n)=B(0,n)$ that
            \[
            \partial_t w^{(k)}\in L^{4/3}(0,T;H^{-1}(B(n))),
            \]
            with a uniform bound. Let us check this. Let $\phi\in C_c^\I(B(n)\times (0,T))$. We have 
            \[
            \langle \partial_t w^{\rev{(k)}}, \phi \rangle =- \langle \nabla w^{(k)},\nabla \phi \rangle - \langle w^{(k)}	\cdot \nabla w^{(k)}	 +V^{(k)}	\cdot \nabla w^{(k)}+w^{(k)}	\cdot \nabla V^{(k)}+V^{(k)}	\cdot \nabla V^{(k)}	, \mathbb P \phi	\rangle.
            \]
            We have 
            \[
            |\langle \nabla w^{(k)},\nabla \phi \rangle|\leq C \| w^{(k)}\|_{\dot H^1} \| \phi\|_{ H^1}.
            \]
            Furthermore, by the Ladyzhenskaya inequality and the divergence free condition,
            \[
            | \langle w^{(k)}	\cdot \nabla w^{(k)},\mathbb P\phi	\rangle |\leq  C(n) \|w^{(k)}\|_{L^2}^{1/2}\|\nabla w^{(k)}\|_{L^2}^{3/2} \|  \phi\|_{ H^1}.
            \]
           This is all standard so far. The perturbed terms are new compared to the classical theory, but they are \textit{more} regular than $w^{(k)}$, so bounds are not an issue here. To confirm this note
           \begin{align*}
           	\langle V^{(k)}	\cdot \nabla w^{(k)} ,\mathbb P \phi \rangle &\leq C \| \nabla w^{(k)}\|_{L^2} \|  \phi \|_{L^6} \|V^{(k)}\|_{L^3}\leq C \| \nabla w^{(k)}\|_{L^2} \|  \phi \|_{H^1}t^{-\frac 3 2 \rev{(\frac 1 p -\frac 1 3)}} \rev{M}.
           \end{align*}
           Additionally, 
                \begin{align*}
           	\langle w^{(k)}	\cdot \nabla V^{(k)} , \mathbb P\phi \rangle &\leq C \| \nabla w^{(k)}\|_{L^2} \| \nabla  \phi \|_{L^2} \|V^{(k)}\|_{L^3}\leq C \| \nabla w^{(k)}\|_{L^2} \|  \phi \|_{H^1}t^{-\frac 3 2 \rev{(\frac 1 p -\frac 1 3)}} \rev{M}.
           \end{align*}
             \rev{Finally, since $\nabla\cdot V^{(k)}=0$ and $\mathbb P$ is bounded on $L^2$ and commutes with derivatives, \eqref{ineq.aaa} gives
           	\begin{align*}
           		|\langle V^{(k)}\cdot \nabla V^{(k)} , \mathbb P\phi \rangle| = |\langle V^{(k)}\otimes V^{(k)} , \nabla\mathbb P\phi \rangle| \leq \|V^{(k)}\|_{L^4}^2\|\nabla\phi\|_{L^2}\leq C  \|  \phi \|_{H^1}\,t^{-3(\frac 1 p -\frac 1 4)} M^2.
           \end{align*}}
           Hence, 
            \begin{align*}
           	\| \partial_t w^{(k)}\|_{H^{-1}(B(n))}(t)&\leq C(n) \|w^{(k)}\|_{L^2}^{1/2}\|\nabla w^{(k)}\|_{L^2}^{3/2}
           	+ C \| \nabla w^{(k)}\|_{L^2}  t^{-\frac 3 2 \rev{(\frac 1 p -\frac 1 3)}} \rev{M} \rev{+ C\,t^{-3(\frac 1 p -\frac 1 4)} M^2}.
           \end{align*}
           We now apply the $L^{4/3}(0,T)$ norm. The leading term on the right-hand side is clearly finite. For the second, we need to bound
           \[
           \int_0^T C \| \nabla w^{(k)}\|_{L^2}^{\frac 4 3}  t^{-\frac 3 2 \rev{(\frac 1 p -\frac 1 3)} \frac 4 3} \,dt.
           \] 
           This will be finite provided $ t^{-\frac 3 2 \rev{(\frac 1 p -\frac 1 3)} \frac {12} 3} \in L^1(0,T)$, which is true for \rev{$p>2$}.            \rev{The third term belongs to $L^{4/3}(0,T)$ because $\frac 4 3\cdot 3\big(\frac 1 p-\frac 1 4\big)=\frac 4 p-1<1$ for $p>2$.}
            These remarks prove the inclusion for $\partial_t w^{(k)}$ uniformly in $k$.
            
            Hence, by \rev{the Aubin--Lions lemma and} a Cantor diagonalization argument, we can extract a sub-sequence which we abusively do not re-label so that 
            \[
            w^{(k)}\buildrel\ast\over\rightharpoonup  w \text{ in }L^\I(0,T; \rev{L^2(\R^3)}) \text{ and }\nabla w^{(k)}\rightharpoonup \nabla w \text{ in }L^2((0,T)\times \R^3),
            \]
         	and 
            \[
            w^{(k)}\to w \text{ in } L^2(0,T;L^2\rev{(B(n))})\rev{\text{ for every }n\in\N}.
            \]
            This can be improved in the standard way to
             \[
            w^{(k)}\to w \text{ in } L^{10/3-}(0,T;L^{10/3-}\rev{(B(n))})\rev{\text{ for every }n\in\N},
            \]
            which is necessary to obtain the local energy inequality.

               \rev{We also have, as shown above, that
   	\[
   	V^{(k)}\to V \text{ uniformly on } B(m)\times[n^{-1},T]\text{ for all } m,n\in\N.
   	\]
   	Next, the bound on $\partial_t w^{(k)}$ passes to the limit; we check this directly, since we have not yet shown that $w$ satisfies an equation. Fix $n\in\N$ and let $\phi\in C_c^\I(B(n)\times(0,T))$. By the uniform bound on $\partial_t w^{(k)}$ in $L^{4/3}(0,T;H^{-1}(B(n)))$,
   	\[
   	\bigg|\int_0^T\int w^{(k)}\cdot\partial_t\phi\,dx\,dt\bigg|\leq C(n)\|\phi\|_{L^4(0,T;H^1_0(B(n)))},
   	\]
   	where $C(n)$ does not depend on $k$. Since $w^{(k)}\to w$ in $L^2(0,T;L^2(B(n)))$, the same inequality holds with $w$ in place of $w^{(k)}$. As $C_c^\I(B(n)\times(0,T))$ is dense in $L^4(0,T;H^1_0(B(n)))$, whose dual is $L^{4/3}(0,T;H^{-1}(B(n)))$, it follows that $\partial_t w\in L^{4/3}(0,T;H^{-1}(B(n)))$ for every $n\in\N$. Together with $w\in L^\I(0,T;L^2)$, this implies that $w$ has a representative which is weakly continuous as an $L^2$-valued map on $[0,T]$---see the discussion surrounding \cite[(3.27)]{KwTs} or \cite[Lemma 3.4]{Tsai-book}---and we use this representative.
   	
   	Moreover, the uniform bound above, which holds for every $t\in(0,T]$, extends to the limit. Write $b(t):=C_p\big[M^{\frac{2p}{4-p}}t^{\frac 1 2\frac{p-2}{4-p}}+t^{\frac 1 2}\big]$ and fix $n\in\N$. Since $w^{(k)}\to w$ in $L^2(0,T;L^2(B(n)))$, after passing to a subsequence, which may depend on $n$, $w^{(k)}(t)\to w(t)$ in $L^2(B(n))$ for all $t$ outside a null set $N_n\subset(0,T)$. Hence, for $t\notin N:=\bigcup_n N_n$ and every $n$,
   	\[
   	\|w(t)\|_{L^2(B(n))}^2=\lim_{k\to\I}\|w^{(k)}(t)\|_{L^2(B(n))}^2\leq b(t),
   	\]
   	and letting $n\to\I$ gives $\|w(t)\|_{L^2}^2\leq b(t)$ for almost every $t\in(0,T)$. Given $t\in(0,T]$, choose $t_j\notin N$ with $t_j\to t$. Then $w(t_j)\rightharpoonup w(t)$ in $L^2$ and, since $b$ is continuous, $\|w(t)\|_{L^2}^2\leq\liminf_{j\to\I}\|w(t_j)\|_{L^2}^2\leq b(t)$. Thus $\|w(t)\|_{L^2}^2\leq b(t)$ for every $t\in(0,T]$. In particular $\|w(t)\|_{L^2}\to0$ as $t\to0^+$, which is condition (5) of Definition \ref{def.weak} for the limit, and, by weak continuity, $w(0)=0$.}

            We need to show that the limit $w$ leads to an $L^{p,\I}$-weak solution $u$.  We know that $w$ belongs to $L^\I(0,T; L^2)\cap L^2(0,T;\dot H^1 )$ by convergence properties. \rev{We have also shown that $w$ is weakly continuous as an $L^2$-valued map on $[0,T]$ with $w(0)=0$. Since $\td u=u-V=w$, this is condition (4) of Definition \ref{def.weak} for the limit.}

            
            The modes of convergence are sufficient to show that $w$ is a distributional solution to the perturbed Navier-Stokes equations, from which it follows that $u$ is a distributional solution to \eqref{eq.ns}.  We do not check the convergence of all terms because the terms from \eqref{eq.ns} are well understood and identical to \cite{BaSeSv}. Instead, we only examine the terms involving $V$ as these could in principle be different. We have for $\phi\in C_c^\I([0,T)\times \R^3)$ that
            \[
            \int_0^t \int (V\cdot \nabla w - V^{(k)}\cdot \nabla w^{(k)})\phi\,dx\,ds=   \int_0^t \int ((V - V^{(k)})\cdot \nabla w^{(k)})\phi +  V \cdot \nabla(w- w^{(k)})\phi\,dx\,ds. 
            \]
            Let $\delta\in (0,T)$ be a small parameter. Then, applying H\"older with the space-time norms $L^2_tL^2_x$ and $L^q_t L^q_x$ and $L^r_tL^r_x$ where $1= \frac 1 2 + \frac 1 q+\frac 1 r$, we have
             \begin{align*}
            &\int_0^\delta \int ((V - V^{(k)})\cdot \nabla w^{(k)})\phi \,dx\,ds 
            \\&\lesssim_M \| \phi \|_{L^r_t L^r_x} \|  w^{(k)}\|_{L^2_t \dot H^1_x} \bigg( \int_0^\delta ( s^{-\frac 3 2 ( 1/p-1/q )}  )^q \bigg)^{1/q},
            \end{align*}
            where we used the fact that $V^{(k)}$ and $V$ solve heat equations with initial data bounded uniformly in $L^{p,\I}$ by $M$.
            We may take $q$ to be arbitrarily close to but greater than $p$. Doing so results in a positive power of $\delta$ when we evaluate the last integral above. Therefore, we can make this term as small as we want by taking $\delta$ small. 
            Additionally, 
            \[
             \int_\delta^t \int ((V - V^{(k)})\cdot \nabla w^{(k)})\phi \,dx\,ds \leq \| \phi \|_{L^2_t \rev{L^2_x}} \| \nabla w^{(k)}\|_{L^2_t \rev{L^2_x}} \|V-V^{(k)}\|_{L^\I( \rev{K}\times (\delta,T))},
            \]
            and the upper bound vanishes as $k\to \I$ for any $\delta$ due to the \rev{local uniform} convergence of $V^{(k)}$\rev{; here $K$ is a compact set containing the spatial support of $\phi$}. It follows that
            \[
             \bigg|\int_0^t \int (V\cdot \nabla w - V^{(k)}\cdot \nabla w^{(k)})\phi\,dx\,ds \bigg|\to 0,
            \]
            as $k\to \I$.  Note that the convergence just shown is  stronger than distributional because we did not put the derivative on $\phi$, but the distributional statement follows from the same argument.
            The other term is bounded by
            \[
             \int_0^t \int | \rev{V} | |w- w^{(k)}| |\nabla \phi|dx\,ds \leq \| w-w^{(k)}\|_{L^2_tL^2_x} \| \nabla \phi\|_{L^r_tL^r_x} \| \rev{V} \|_{L^q_{t}L^q_x}.
             \]
             Again, $q$ can be chosen close to $p$ so that $\| \rev{V}\|_{L^q_{x,t}}$ is finite. Then, this term vanishes by the strong convergence of $w^{(k)}$ in $L^2_tL^2_x$ \rev{on the (compact) support of $\phi$}.

             The other terms that must be shown to vanish as $k\to \I$ are 
            \[
              \int_0^t \int (w\cdot \nabla V - w^{(k)}\cdot \nabla V^{(k)})\phi\,dx\,ds,
            \]
            and 
            \[
              \int_0^t \int (V\cdot \nabla V - V^{(k)}\cdot \nabla V^{(k)})\phi\,dx\,ds.
            \]
            For the first, to get a distributional solution the derivative is moved to $\phi$ after which the same argument as above goes through. For the second of these, a variation of this also works but, since $V^2$ is more singular than $V$,  we check the details explicitly. It suffices to show that 
            \[
            \int_0^t\int |V^{(k)} - V|(|V^{(k)}|+|V|) |\nabla \phi| \,dx\,ds\to 0.
            \]
            This is easy on $(\delta, T)$ \rev{by the local uniform convergence of $V^{(k)}$}, so let us consider $(0,\delta)$. We have 
           \begin{align*}
           	 \int_0^\delta\int |V^{(k)} - V|(|V^{(k)}|+|V|) |\nabla \phi| \,dx\,ds&\leq \int_0^\delta \big( \|  V\|_{L^q_{x}}^2 +  \|  V^{(k)}\|_{L^q_{x}}^2 \big) \| \nabla \phi\|_{L^r_{x}}\,ds,
           \end{align*}
           where now $\frac 1 r +\frac 2 {q}=1$. We can again choose $q$ close to $p$ so that the above is bounded by a positive power of $\delta$. These remarks show that $w$ is a distributional solution to perturbed Navier-Stokes.

	\rev{For the pressure, we  take $\pi^{(k)}=R_iR_j(u^{(k)}_iu^{(k)}_j)$, where $R_i$ are the Riesz transforms (see the discussion following Theorem \ref{thrm.stability}). Writing $u^{(k)}\otimes u^{(k)}=V^{(k)}\otimes V^{(k)}+V^{(k)}\otimes w^{(k)}+w^{(k)}\otimes V^{(k)}+w^{(k)}\otimes w^{(k)}$, the estimate \eqref{ineq.aaa} and the uniform bounds on $w^{(k)}$ show that $u^{(k)}\otimes u^{(k)}$ is bounded in $L^{3/2}+L^{5/3}$ on $\R^3\times(0,T)$ uniformly in $k$. Here we have used $\|V^{(k)}\otimes V^{(k)}\|_{L^{3/2}_x}\leq \|V^{(k)}\|_{L^3}^2$ and $\|V^{(k)}\otimes w^{(k)}\|_{L^{3/2}_x}\leq\|V^{(k)}\|_{L^6}\|w^{(k)}\|_{L^2}$, as well as $p>2$. Moreover, $u^{(k)}\otimes u^{(k)}\to u\otimes u$ in $L^1_\loc(\R^3\times(0,T))$ because, for every $n\in\N$, $w^{(k)}\to w$ in $L^2(0,T;L^2(B(n)))$ and $V^{(k)}\to V$ uniformly on $B(n)\times[n^{-1},T]$. Hence $u^{(k)}\otimes u^{(k)}\rightharpoonup u\otimes u$ weakly in $L^{3/2}(\R^3\times(0,T))+L^{5/3}(\R^3\times(0,T))$. Since $R_iR_j$ is bounded on $L^{3/2}$ and on $L^{5/3}$, it follows that $\pi^{(k)}\rightharpoonup\pi:=R_iR_j(u_iu_j)$ weakly in the same space. In particular, $\pi\in L^{3/2}_\loc(\R^3\times(0,T))$.
		
		We next check that the pressure term in \eqref{ineq:CKN-LEI.p} converges. Let $\phi$ and $t$ be as in \eqref{ineq:CKN-LEI.p}. Then $\supp\phi\subset B(n)\times[n^{-1},T]$ for some $n\in\N$, and we let $K:=B(n)\times(n^{-1},T)$. Writing $u^{(k)}-u=(V^{(k)}-V)+(w^{(k)}-w)$ and using H\"older's inequality on $K$, we have
		\begin{align*}
			&\Big|\int_0^t\int\pi^{(k)}u^{(k)}\cdot\nabla\phi\,dx\,ds-\int_0^t\int\pi u\cdot\nabla\phi\,dx\,ds\Big|
			\\&\leq \Big|\int_0^t\int\pi^{(k)}(u^{(k)}-u)\cdot\nabla\phi\,dx\,ds\Big|+\Big|\int_0^t\int(\pi^{(k)}-\pi)\,u\cdot\nabla\phi\,dx\,ds\Big|
			\\&\leq \|\pi^{(k)}\|_{L^{3/2}(K)}\|\nabla\phi\|_{L^\I}\Big(|K|^{\frac 1 3}\|V^{(k)}-V\|_{L^\I(K)}+\|w^{(k)}-w\|_{L^3(K)}\Big)
			\\&\quad+\Big|\int_0^t\int(\pi^{(k)}-\pi)\,u\cdot\nabla\phi\,dx\,ds\Big|.
		\end{align*}
		We now show that the right-hand side vanishes as $k\to\I$. Since $\pi^{(k)}$ is bounded in $L^{3/2}+L^{5/3}$ uniformly in $k$ and $K$ is bounded, $\|\pi^{(k)}\|_{L^{3/2}(K)}$ is bounded uniformly in $k$. Furthermore, $V^{(k)}\to V$ uniformly on $K$, and $w^{(k)}\to w$ in $L^3(K)$ since $3<10/3$. This takes care of the first term. For the second term, note that $\mathbf 1_{(0,t)}(s)\,u\cdot\nabla\phi$, extended by zero, belongs to $L^3\cap L^{5/2}$, which is the dual of $L^{3/2}+L^{5/3}$. The second term therefore vanishes in the limit by the weak convergence of $\pi^{(k)}$.}

            The local energy inequality is easy to prove and we omit the details. We do note that some care needs to be taken because there is no cancellation when integrating $w\cdot\nabla V w\phi$. However, if $\phi \in C_c^\I ( (0,T)\times \R^3)$, then we are away from the singularity of $V$ and so this term becomes more regular than terms which have the cancellation and which are therefore known to converge via classical works.  
             
            \end{proof}

            \begin{proof}[Proof of Theorem \ref{thrm.existence}]
            		We  now construct $L^{p,\I}$-weak solutions.  Note that $C_{c,\sigma}^\I$ is dense in the weak-star topology in $L^{p,\I}$. Taking $u_0^{(k)}\in C_{c,\sigma}^\I$ as our initial data, we have the existence of a global-in-time suitable Leray weak solution with data $u_0^{(k)}$ by classical results \cite{leray,CKN}. See also \cite[Definition 3.1 \& Theorem 3.9]{Tsai-book}. These solutions are  $L^{p,\I}$-weak solutions for obvious reasons\rev{; in particular, they attain their data strongly in $L^2$ and so satisfy condition (5) of Definition \ref{def.weak}}. Applying Theorem \ref{thrm.stability} completes the proof.
            \end{proof}

          \subsection{\textit{A priori} bounds revisited}
           
          In this section we extend the \textit{a priori} bound in Theorem \ref{thrm.aprioribound} to other space-time norms. 
          If $u_0\in L^{p,\I}$, then the decay property of $L^{p,\I}$-weak solutions states that 
          \[
          \| u - P_0\|_{L^2}\lesssim t^{\sigma},
          \]
          for   appropriate choices of $\sigma=\sigma(p)$. There is a dimensional relationship between the suppressed constant and the choice of $\sigma$. In the case that $p=3$ the dimensional relationship is that the suppressed constant is dimensionless.  In that case it is possible to extend the above results to space-time spaces with different dimensions. The following lemma is from \cite{BP2}.

          \begin{lemma}Fix $q\in (3/2 ,3)$, $T>0$ and $k\in \N_0$. Assume $u_0\in L^{3,\I}$ and is divergence free. Let $u$ be an $L^{3,\I}$-weak solution with initial data $u_0$. Then, letting $r =\frac {2q} {2q-3}$,
         	\[
         	\| u-P_0 \|_{L^r(0,T;L^{q})} \lesssim_{k,q,u_0} T^{\frac 1 2}.
         	\]
          \end{lemma} 
          To get a better sense of the scaling of the preceding estimate, it is helpful to note the excluded endpoint for this estimate would give
          \[
          \| u - P_0\|_{L^{3/2}}\lesssim t^{1/2},
          \]
          so this can be seen as shifting the \textit{a priori} bound from $L^2$ to $L^{3/2}$.

          We presently explore the extent to which these remarks extend to the supercritical setting. 
          The updated lemma is as follows.
          \begin{proposition}\label{prop.LpinftyDecaySpacteTime}
         	Fix $q\in (3/2 ,3)$, $T>0$ and $k\in \N_0$. Assume $u_0\in L^{p,\I}$ and is divergence free. Let $u$ be a $L^{p,\I}$-weak solution with initial data $u_0$. Then,  letting 	$r = \frac {2q} {2q-3}$,
         	\[
         	\| u-P_0 \|_{L^r(0,T;L^{q})} \lesssim_{k,q,u_0}  	 T^{  \sigma(p) } .
         	\]
          \end{proposition}

          The proof will use mild solution techniques. 
          Recall that $e^{t\Delta}\mathbb P  $ is the Oseen tensor. Let $S$ denote its kernel.  We have the following pointwise estimate for $S$,
          \EQ{\label{ineq.oseen}
          	|\pd_t^m\nb_x^k S(x,t)|\le \frac{C_{k,m}}{(|x|+\sqrt t)^{3+k+2m}},}
          see the original work of Oseen \cite{Oseen} as well as \cite{VAS} and the more recent references \cite{LR,MaTe,ShSh}.  This estimate ultimately leads to a class of $L^p$-$L^q$ estimates \cite{FJR,Kato,yamazaki,LR,Tsai-book}:
          \[
          \| D^\alpha e^{t\Delta}\mathbb P F \|_{L^p(\R^3)}\lesssim \frac 1 {t^{\frac {|\alpha|} 2 + \frac 3 2 (\frac 1 q-\frac 1 p)}} \| F\|_{L^q(\R^3)},
          \]
          where $\alpha$ is a multi-index.
          
          \begin{proof} 
         	Observe that 
         	\[
         	u-P_0 = B(u,u-P_0)+ B(u-P_0,P_0)+ {B(P_0,P_0)}.
         	\]
         	and 
         	\[
         	\| P_0\|_{L^{2q}}\lesssim t^{3/(4q)-3/(2p)}.
         	\]
         	The last inequality allows us to quickly conclude that 
         	\[
         	\|B(P_0,P_0)(t)\|_{L^q} \lesssim t^{ \frac 1 2 -\frac 3 p +\frac 3 {2q} }.
         	\]
         	Hence,
         	\[
         	\bigg\|				\|B(P_0,P_0)(s)\|_{L^q}  \bigg\|_{L^r(0,T)}\lesssim T^{ \frac 1 2 -\frac 3 p +\frac 3 {2q} +\frac 1 r}.
         	\]

         	For the other terms we have that  
         	\EQ{ 
         		\| B(u-P_0,P_0) \|_{L^{q}}&\lesssim \int_0^t \frac 1 {(t-s)^\frac 1 2}\big( \| u-P_0\|_{L^{2q}}^2(s) + \| P_0\|_{L^{2q}}^2(s)\big)\,ds
         	}
         	The latter term is bounded as was $B(P_0,P_0)$. For the former  we note that 
         	\[
         	\|u-P_0\|_{L^{\rev{2q}}}\leq C \| \nabla (u-P_0)\|_{L^2}^{3/2-3/(2q)} \| u-P_0\|_{L^2}^{3/(2q)-1/2},
         	\]
         	by the Gagliardo-Nirenberg inequality\rev{, which is applicable because $2q\in(3,6)\subset[2,6]$}.  Hence
         	\EQ{ 
         		\| B(u-P_0,P_0) \|_{L^{q}}
         		&\lesssim \int_0^t \frac 1 {(t-s)^\frac 1 2}   \| \nabla (u-P_0)\|_{L^2}^{3-3/q} \| u-P_0\|_{L^2}^{3/q-1}(s)\,ds
         		\\&\lesssim t^{(3/q-1)\sigma(p) /2  } \int_0^t \frac 1 {(t-s)^\frac 1 2}  \| \nabla (u-P_0)\|_{L^2}^{3-3/q}  \,ds
         	}
         	Now, $3-3/q\in (3/2,2)$ provided $q\in (3/2,3)$. Hence 
         	$
         	\frac {2q} {3q-3}\in (1,2), 
         	$
         	the H\"older conjugate of which is $2q/(3-q)\in (2,\I)$.
         	We, therefore, cannot apply H\"older's inequality in the time-variable here as it would lead to a divergent integral 
         	\[
         	\int_0^t \frac 1 {(t-s)^{ \frac 1 2 ( 2q / (3-q))}}\,ds.
         	\]
         	If we would like an upper bound, therefore, we must use space-time norms in conjunction with the Hardy-Littlewood-Sobolev inequality---the application of the latter is why $L^\I(0,T;L^{3/2})$ is excluded. 
         	In particular we have 
         	\[
         	\bigg\| I_{1/2} [ \| \nabla (u-P_0) \|_{L^2}^{3-3/q}(s) \chi_{(0,T)}(s)] \bigg\|_{L^r(\R)}\lesssim_r 	\bigg\|  \| \nabla (u-P_0) \|_{L^2}^{3-3/q}(s) \chi_{(0,T)}(s) \bigg\|_{L^{\td r}(\R)},
         	\]
         	where 
         	\[
         	\frac 1 r = \frac 1 {\td r} -\frac 1 2.
         	\]
         	We take $\td r = 2q /(3q-3)\in (1,2)$, which gives 
         	\[
         	r = \frac {2q} {2q-3}.
         	\]
         	Then,
         	\[
         	\bigg\|  \| \nabla (u-P_0) \|_{L^2}^{3-3/q}(s) \chi_{(0,T)}(s) \bigg\|_{L^{\td r}(\R)} \lesssim \bigg(\int_0^T \| \nb (u-P_0)\|_{L^2}(s)^2 \,ds \bigg)^{\frac {3q-3} {2q}}\lesssim T^{\frac {\sigma (p)} 2\frac {3q-3} {q}},
         	\]
         	by the decay estimate for $L^{p,\I}$-weak solutions. Hence
         	\[
         	\| B(u-P_0,u-P_0)\|_{L^r(0,T;L^q(\R^3))} \lesssim T^{\sigma(3/q-1) /2 + \sigma (3q-3)/(2q) }+ T^{ \frac 1 2 -\frac 3 p +\frac 3 {2q} +\frac 1 r}=T^{\sigma(p)}+ T^{ \frac 1 2 -\frac 3 p +\frac 3 {2q} +\frac 1 r}.
         	\]
         	Note that  $ \frac 1 2 -\frac 3 p +\frac  {3} {2q} +\frac {2q-3} {2q}=  \frac 3 2 -\frac 3 p$. Furthermore, if $p\in (2,3)$, one can check that 
         	\[ 
         \sigma(p)=	\frac 1 2 \frac {p-2}{4-p} < \frac 3 2 - \frac 3 p,
         	\]
         	with equality holding at $p=2,\,3$.
          \end{proof}	 

\begin{revblock}
\subsection{Restarting $L^{p,\I}$-weak solutions at later times}
In Section \ref{sec.reg} we apply Corollary \ref{cor:apriori-bound} with the initial time moved to a later time. This is not immediate from Definition \ref{def.weak}, because condition (5) is only imposed at the initial time and, as for Leray solutions, strong right-continuity in $L^2$ at a prescribed later time is not available in general. The following lemma shows that it is available at almost every time.

\begin{lemma}\label{lemma.restart}
Let $2<p\leq 3$ and let $(u,\pi)$ be an $L^{p,\I}$-weak solution on $[0,T]$ with data $u_0$. Let $S\subset (0,T)$ be the set of Lebesgue points of the $L^2$-valued map $\tau\mapsto \td u(\tau)$, which has full measure in $(0,T)$. If $s\in S$ and $u(\cdot,s)\in L^{p,\I}$, then $(u_s,\pi_s)(x,\tau):=(u,\pi)(x,s+\tau)$ is an $L^{p,\I}$-weak solution on $[0,T-s]$ with data $u(\cdot,s)$. Consequently, for all $\tau\in(0,T-s]$,
\EQ{\label{ineq.restart}
\|u(s+\tau)-e^{\tau\Delta}u(s)\|_{L^2}^2\leq C_p\left[\|u(s)\|_{L^{p,\I}}^{\frac{2p}{4-p}}\tau^{\sigma(p)}+\tau^{\frac 1 2}\right].
}
\end{lemma}

\begin{proof}
Fix $s\in S$ with $u(s)\in L^{p,\I}$ and let $\td u_s(\tau):=u_s(\tau)-e^{\tau\Delta}u(s)$. Since $u(s)=e^{s\Delta}u_0+\td u(s)$ and $e^{\tau\Delta}e^{s\Delta}=e^{(s+\tau)\Delta}$,
\EQ{\label{eq.restart}
\td u_s(\tau)&= u_s(\tau) -e^{\tau \Delta}u(s)
=u(s+\tau) -e^{\tau \Delta}(   \td u(s)  + e^{s\Delta}u_0 )
=\td u(s+\tau) -e^{\tau \Delta} \td u(s).
}
We check the conditions of Definition \ref{def.weak} for $(u_s,\pi_s)$. Condition (1) follows from \eqref{eq.restart}, since $\td u(s+\cdot)$ satisfies \eqref{ineq.BSSbound.p} and the caloric extension of $\td u(s)\in L^2$ satisfies the energy equality for the heat equation while clearly $\pi_s\in L^{3/2}_\loc$ is unchanged. Condition (2) is clear. For condition (3), if $\phi\in C_c^\I(\R^3\times(0,T-s])$ is non-negative, then $\phi(\cdot,\cdot-s)$, extended by zero, is a non-negative element of $C_c^\I(\R^3\times(0,T])$ which vanishes for times in $(0,s]$, and \eqref{ineq:CKN-LEI.p} for $u$ with this test function is exactly the local energy inequality for $u_s$ with the test function $\phi$. Condition (4) follows from \eqref{eq.restart}, condition (4) for $u$ and the continuity of the heat semigroup on $L^2$; note that $\td u_s(0)=0$.

It remains to verify condition (5), that is, $\td u_s(\tau)\to 0$ in $L^2$ as $\tau\to0^+$. Since $e^{\tau\Delta}\td u(s)\to \td u(s)$ in $L^2$, by \eqref{eq.restart} it suffices to show that $\td u(s+\tau)\to\td u(s)$ strongly in $L^2$ as $\tau\to 0^+$. Decompose $u_0=\bar u_0^N+\td u_0^N\in L^4+L^2$ as in Lemma \ref{lemma.split} with any fixed $N>0$, and let $V=e^{t\Delta}\bar u_0^N$ and $w=u-V=\td u+e^{t\Delta}\td u_0^N$, as in Lemma \ref{lemma.energy}. Since $e^{t\Delta}\td u_0^N$ is continuous into $L^2$, $s$ is also a Lebesgue point of $w$, so \eqref{ineq.energy.s} holds with this $s$. As the integrals on the right-hand side of \eqref{ineq.energy.s} over $(s,s+\tau)$ vanish as $\tau\to0^+$, we obtain
\[
\limsup_{\tau\to 0^+}\|w(s+\tau)\|_{L^2}\leq \|w(s)\|_{L^2}.
\]
On the other hand, $w(s+\tau)\rightharpoonup w(s)$ weakly in $L^2$ as $\tau\to 0^+$ by condition (4), since $w=\td u+e^{t\Delta}\td u_0^N$. Hence $w(s+\tau)\to w(s)$ strongly in $L^2$, and therefore $\td u(s+\tau)=w(s+\tau)-e^{(s+\tau)\Delta}\td u_0^N\to w(s)-e^{s\Delta}\td u_0^N=\td u(s)$ strongly in $L^2$. This proves that $(u_s,\pi_s)$ is an $L^{p,\I}$-weak solution with data $u(s)$, and \eqref{ineq.restart} now follows from Corollary \ref{cor:apriori-bound}.
\end{proof}

The preceding lemma motivates the following definition, which is used in Section \ref{sec.reg} for solutions defined on a time interval containing the initial time $0$ in its interior.

\begin{definition}[$L^{p,\I}$-weak solutions on an open time interval]\label{def.weakinterval}
Let $2<p\leq3$ and $-\I<a<b<\I$. We say that $u$ and an associated pressure $\pi$ comprise an $L^{p,\I}$-weak solution on $\R^3\times(a,b)$ if $\sup_{a<t<b}\|u(\cdot,t)\|_{L^{p,\I}}<\I$ and if, for almost every $s\in(a,b)$ and every $b'\in(s,b)$, the time translate $(u,\pi)(\cdot,s+\cdot)$ is an $L^{p,\I}$-weak solution on $[0,b'-s]$ with data $u(\cdot,s)$ in the sense of Definition \ref{def.weak}.
\end{definition}

We make two comments on this definition. First, it asks no more at interior times than what is inherited from a solution started at the left endpoint: if $\sup_{a<t<b}\|u(\cdot,t)\|_{L^{p,\I}}<\I$ and, for every $b'\in(a,b)$, $(u,\pi)(\cdot,a+\cdot)$ is an $L^{p,\I}$-weak solution on $[0,b'-a]$ with data $u(\cdot,a)$, then, by Lemma \ref{lemma.restart}, $(u,\pi)$ is an $L^{p,\I}$-weak solution on $\R^3\times(a,b)$. Second, in accordance with Lemma \ref{lemma.restart}, the restarting property is only required at almost every time; in particular, nothing is assumed at any prescribed interior time such as a singular time. 
\end{revblock}

	 \section{Asymptotic expansion in the time variable}\label{sec.asy}
	In this section we prove Theorem \ref{thrm.asympt}. We begin by recalling  the local smoothing result of Jia and \v Sver\'ak \cite{JS}. Note that $L^2_\uloc$ is the space of uniformly locally square integrable functions and is defined by the norm
	 \[
	 \|f\|_{L^2_\uloc}^2 :=\sup_{x_0\in \R^3}\int_{B_1(x_0)}|f|^2\,dx. 
	 \]
	 Let $E^2$ denote the closure of $C_c^\I$ in the $L^2_\uloc$ norm. For $2<p$,
	 $L^{p,\I}$ embeds in $E^2$. See the appendix of \cite{BT8} for the proof of this in the $p=3$ case. The general case follows from the same argument. 
	 Local smoothing as presented below refers to local energy solutions (a.k.a.~local Leray solutions using the terminology of \cite{JS}; see also \cite{KS,LR}). 
\begin{revblock}
	 We now verify that $L^{p,\I}$-weak solutions, $2<p\leq 3$, are local energy solutions in the sense of \cite{JS} (see also \cite{KS,LR}), so that local smoothing  applies to them. Let $u$ be an $L^{p,\I}$-weak solution on $[0,T]$ with data $u_0$ and write $u=V+\td u$ with $V=e^{t\Delta}u_0$. Since $L^{p,\I}\subset E^2\subset L^2_\uloc$, the caloric extension $V$ satisfies the standard local energy bound
	 \[
	 \sup_{0<t<T}\|V(t)\|_{L^2_\uloc}^2+\sup_{x_0\in\R^3}\int_0^T\int_{B_1(x_0)}|\nb V|^2\,dx\,dt\lesssim_T\|u_0\|_{L^2_\uloc}^2,
	 \]
	 which follows by testing the heat equation against $V$ times the square of a cut-off function. Together with condition (1) of Definition \ref{def.weak}, this gives the required local energy bounds for $u$. The pressure is given by $\pi=R_iR_j(u_iu_j)$, where $R_i$ are the Riesz transforms. Writing $u\otimes u=V\otimes V+V\otimes \td u+\td u\otimes V+\td u\otimes \td u$, we have $V\otimes V,\,V\otimes\td u\in L^{3/2}(\R^3\times(0,T))$ by \eqref{ineq.aaa}, using $\|V\otimes V\|_{L^{3/2}}\leq\|V\|_{L^3}^2$, $\|V\otimes\td u\|_{L^{3/2}}\leq\|V\|_{L^6}\|\td u\|_{L^2}$ and $p>2$, while $\td u\otimes\td u\in L^{5/3}(\R^3\times(0,T))$ by condition (1) and interpolation. Hence $\pi\in L^{3/2}_\loc(\R^3\times[0,T))$, and the local pressure expansion required in the definition of local energy solutions holds; see \cite{LR,BT7}. The distributional formulation and the local energy inequality are conditions (2) and (3) of Definition \ref{def.weak}, and the weak continuity of $u$ in $L^2_\loc$ follows from condition (4) and the continuity of $t\mapsto V(t)$ in $L^2_\loc$ on $[0,T]$, which holds since $u_0\in E^2$. The initial data is attained in $L^2_\loc$ because, for any compact $K\subset\R^3$,
	 \[
	 \|u(t)-u_0\|_{L^2(K)}\leq \|\td u(t)\|_{L^2}+\|e^{t\Delta}u_0-u_0\|_{L^2(K)}\to 0\text{ as }t\to0^+,
	 \]
	 by condition (5) of Definition \ref{def.weak} and since $u_0\in E^2$. We finally check that for every $R>0$, $\int_0^T\int_{B_R(x_0)}|u|^2\,dx\,dt\to 0$ as $|x_0|\to\I$. This holds for $\td u$ because $\td u\in L^\I(0,T;L^2)$ and it holds for $V$ because $u_0\in E^2$. Because $u=V+\td u$, the same decay property follows for $u$.
\end{revblock}
	 
	 \begin{theorem}[Local smoothing {\cite[Theorem 3.1]{JS}}]\label{theorem:JSlocalsmoothing} 
	 	Let $u_0\in E^2$ be divergence free. Suppose $u_0|_{B_2(0)}\in L^p(B_2(0))$ with $\|u_0\|_{L^p(B_2(0))}<\I$ and $p>3$. Decompose $u_0 =U_0+U_0'$ with $\div U_0 =0,\,U_0|_{B_{4/3}} = u_0,\, \operatorname{supp} U_0 \Subset B_2(0)$ and $\|U_0\|_{L^p(\R^3)}<C(p,\|u_0\|_{L^p(B_2(0))})$. Let $U$ be the locally-in-time defined mild solution to \eqref{eq.ns} with initial data $U_0$. Then, there exists a positive $T=T(p,\|u_0\|_{L^2_\uloc},\|u_0\|_{L^p(B_2(0))})$ such that any local energy solution $u$ with data $u_0$ satisfies 
	 	\EQ{\label{ineq.a}
	 		\|u-U\|_{C^\gamma_{\operatorname{par}}(\overline {B}_{\frac{1}{2}}\times [0,T])}\le C(p,\|u_0\|_{L^2_\uloc},\|u_0\|_{L^p(B_2(0))}),
	 	} 
	 	for some $\gamma=\gamma(p)\in(0,1)$.
	 \end{theorem}
	 
	 See also \cite{BP2020,KMT,KMT2,Kwon,ABP} for more recent work on local smoothing which allows locally critical data which is also locally small; the above statement on the other hand is for locally sub-critical data. The dependence on $\|u_0\|_{L^2_\uloc}$ can be replaced with $\|u_0\|_{L^{p,\I}}$ because $L^{p,\I}$ embeds in $L^2_\uloc$ for $2<p$.
	 Additionally, the result can be re-stated for any ball $B\Subset B_2(0)$ replacing $B_{4/3}(0)$ with the understanding that the constant in \eqref{ineq.a} will change accordingly.

	 Our proof is  structured in several steps. In each step, we extract additional terms for the  asymptotic expansion.
	 
	 \bigskip \noindent \underline{Zeroth step}:
	 By assumption, the local smoothing of Jia and \SVERAK~\cite{JS} applies in $B_0\Subset B$. In particular, there exists $T$ so that $u-a \in C^\gamma_{\operatorname{par}}(B_0\times [0,T]  )$ for some $\ga>0$ where $a$ represents a strong solution to the Navier-Stokes equations coming from an initial data that is bounded, supported in $B$ and is identical to $u_0$ on a ball  $B_0'$ for which $B_0\Subset B_0' \Subset B$. All balls are taken to be concentric. Call the localized initial data $a_0$. We will work with a nested sequence of balls satisfying $B_k\Subset B_{k-1}$ for $k=1,2$.	The final ball in the iterative procedure should be the ball $B_\Omega$ from the statement of Theorem \ref{thrm.asympt}. Note that the time-scale in Theorem \ref{thrm.asympt} is exactly the $T$ coming from local smoothing.
	 
	 We record a technical lemma which extends the local regularity properties of $u-a$ to $u-P_0$. 
	 \begin{lemma}\label{lemma.heat} Assume that $u_0\in L^{p,\I}$ for $p\in (2,3)$ and $u_0|_{B}\in L^q(B)$ for some $q\in (3,\I]$. Suppose $u$ is an $L^{p,\I}$-weak solution with data $u_0$. Then, $u-P_0 \in  C^\gamma_{\operatorname{par}}(B_0\times [0,T] )$ for some $\gamma\in (0,1/2)$
	 	and 
	 	\[
	 	| u- P_0|(x,t)\lesssim_{u_0,q,p}  t^{\frac \ga 2} + t^{\frac 1 2 - \frac 3 {2q}}.
	 	\]
	 	Additionally,
	 	\[
	 	\|	P_0\|_{L^\I(0,T;L^q(B_0))}\lesssim_{B_0,B,p,q,u_0} 1
	 	\]
	 \end{lemma}
	 Note that in Theorem \ref{thrm.asympt} we only consider $q=\I$ but this lemma applies to any $q\in (3,\I]$. As is elaborated on in Remark \ref{remark.generalCase}, Theorem \ref{thrm.asympt} can be generalized to the case of $q<\I$ using the full statement included here.
	 
	 \begin{proof}  	The details of this when $p=3$ are worked out at the beginning of  \cite[Proof of Theorem 1.3]{BP2}, and some of these observations are directly applicable here. 
	 	We can write $u-P_0$ as 
	 	\[
	 	u-P_0 = u-a +a-e^{t\Delta}a_0+e^{t\Delta}a_0-P_0.
	 	\]
	 	We have $ u-a  \in  C^\gamma_{\operatorname{par}}(B_0\times [0,T]  )$ by assumption while $a-e^{t\Delta}a_0 \in C^\gamma_{\operatorname{par}}(B_0\times [0,T]  )$ by \cite[(3.3)]{BP2}. The remaining term  is
	 	\[
	 	e^{t\Delta}a_0-P_0 = e^{t\Delta}(a_0-u_0). 
	 	\]
	 	We only care about $x\in B_0$ and note that $a_0-u_0=0$ in $B_0'$. Then 
	 	\EQ{\label{heat.helpful}
	 		e^{t\Delta}(a_0-u_0)(x)&\lesssim \frac 1 {t^{3/2}} \int_{y\notin B_0'} e^{-| x-y |^2/(4t)}  |a_0-u_0|\,dy
	 		\\& \lesssim t^{-3/2} \| a_0-u_0\|_{L^{p,\I}} \bigg\|  \|		e^{-| x-y |^2/(4t)} (1-\chi_{B_0'})	\|_{L^{ p', 1}}	\bigg\|_{L^\I (B_0)}
	 	}
	 	where $p'$ is the H\"older conjugate of $p$. Let $r$ be the distance between the boundaries of $B_0$ and $B_0'$. 
	 	We claim that 
	 	\EQ{\left\|\|e^{-\frac{|x-y|^2}{4t}}(1-\chi_{B_0'})\|_{L^{p',1}_y}\right\|_{L^\I_x(B_0)}\lesssim_{r,p} e^{\frac{-r^2}{4t}}.}
	 	For simplicity take the balls to be centered at $0$. Let $R>0$ satisfy $B_0'  = B(0,R)$. 
	 	Then, recalling that $x\in B_0$,
	 	\EQ{\|e^{-\frac{|x-y|^2}{4t}}(1-\chi_{B_0'})\|_{L^{p',1}_y}
	 		&= p' \int_0^\I \mu \left\{ y:e^{-\frac{|x-y|^2}{4t}}(1-\chi_{B_0'}(y)) \ge s\right\}^{\frac 1 {p'}}\,ds,
	 	}
	 	where $\mu$ is Lebesgue measure. 
	 	Note that the above set can be written as \[A(x,s)=\{y:|x-y|\le \sqrt{-4t\ln(s)},\, |y|>R\}=B(x,(-4t\ln(s))^{\frac 1 2}) \setminus B_0',\] which is well-defined because $t\ge 0$ and $0< s\le e^{-r^2/(4t)}<1$. Then, \EQ{\left\|\|e^{-\frac{|x-y|^2}{4t}}(1-\chi_{B_0'}(y))\|_{L^{p',1}_y}\right\|_{L^\I_x(B_0)}
	 			&\lesssim \left\|\int_0^\I 
	 			\mu (A(x,s))^{\frac 1 {p'}}\, ds\right\|_{L^\I_x(B_0)}\\
	 			&\lesssim \int_0^{e^{\frac{-r^2}{4t}}} |-4t\ln(s)|^{3/(2p')} \, ds.
	 			} 
	 	Note that $ 1 < |4t\ln (s)/r^2|<\I$. Then
	 	\[
	 	\int_0^{e^{\frac{-r^2}{4t}}} |-4t\ln(s)|^{3/(2p')} \, ds\leq  	\int_0^{e^{\frac{-r^2}{4t}}}  r^{3/p'} \frac {|4t\ln(s)|}{r^2}  \, ds\lesssim_{r,T} e^{\frac{r^2}{4t}}.
	 	\]
	 	Hence,
	 	\EQ{\label{helpful2}
	 		|e^{t\Delta}(a_0-u_0)(x)|&\lesssim_{r,T}  t^{-3/2}   e^{\frac{-r^2}{4t}}\| a_0-u_0\|_{L^{p,\I}}   \lesssim_{r,T} t^{\frac \ga 2} \| a_0-u_0\|_{L^{p,\I}} ,
	 	} 
	 	where we used the fact that the exponential part of the pre-factor decays rapidly as $t\to 0^+$. Since $T=T(u_0)$, we can ignore the dependence on $T$. 
	 	
	 	We now prove the statement about $ \|	P_0\|_{L^\I(0,T;L^q(B_0))}$. Note that 
	 	\[
	 	P_0  = e^{t\Delta} (  u_0\chi_{B} ) + e^{t\Delta} (  u_0(1-\chi_{B}) ).
	 	\]
	 	we plainly have 
	 	\[
	 	\|  e^{t\Delta} (  u_0\chi_{B} )  \|_{L^\I(0,T;L^q(\R^3))}\leq  \| u_0 \chi_{B} \|_{L^q}.
	 	\]
	 	On the other hand, for $x\in B_0$, the above estimate for $e^{t\Delta}(a_0-u_0)$ implies  
	 	\[
	 	e^{t\Delta} (  u_0(1-\chi_{B}) ) \in L^\I (0,T;L^\I(B_0)),
	 	\]
	 	which implies the advertised inclusion.
	 \end{proof}

	We conclude the zeroth step with the observation that, by Theorem \ref{theorem:JSlocalsmoothing} and Lemma \ref{lemma.heat},
	 \[
 {	 \sup_{x\in B_0} |u-P_0|(x,t)\lesssim   t^{\min \{  \frac \ga 2, \frac 1 2 \}}=  t^{\frac \ga 2},\quad t\in (0,T).}
	 \]

	 \bigskip \noindent \underline{First step}:
	 It is convenient to re-define $B$ as
	 \[
	 B(f,g) = -\int_0^t e^{(t-s)\Delta} \mathbb P\nabla \cdot \bigg(\frac 1 2 f\otimes g + \frac 1 2 g\otimes f   \bigg),
	 \]
	 which is a symmetric operator. We still have $P_k = P_0 +B(P_{k-1},P_{k-1})$ and $u=P_0 +B(u,u)$. We take this as our definition for the rest of the proof and will then convert to the non-symmetric operator at the end of the proof to match the notation in the introduction. 
	 Since $u$ is mild we have
	 \[
	 u-P_1= u-P_0 -B(P_0,P_0) = B(u,u)-B(P_0,P_0)=B(u-P_0,u-P_0) + 2B(P_0,u-P_0).
	 \]
	 Let $B_1$ be a ball nested inside $B_0$, with the distance between the boundaries of $B_0$ and $B_1$ being $r_1>0$.

	 We bound $B(u-P_0,u-P_0)$ and $2B(P_0,u-P_0)$ separately. First let  $\chi_k= \chi_{B_k}$ with $k=1$ (the same convention will apply to   $B_2$ and $B_3$ later) and consider
	 \[
	 B(u-P_0,u-P_0)(x,t) = B(u-P_0,(u-P_0)\chi_0)(x,t) + B(u-P_0,(u-P_0)(1-\chi_0))(x,t),
	 \]
	 where $(x,t)\in B_1 \times [0,T]$.
	 We have by the billinear estimate that, for the near-field term,
	 \[
	 |B(u-P_0,(u-P_0)\chi_0)(x,t)|\lesssim  \int_0^t \frac 1 {(t-s)^{1/2}} s^{2\ga} \,ds \lesssim t^{1/2 + 2\gamma }.
	 \]
	 On the other hand for the far-field part we have 
	 \[
	 |B(u-P_0,(u-P_0)(1-\chi_0))(x,t)| \leq \int_0^t\int_{y\notin B_0}\frac 1 {(|x-y|+\sqrt{t-s})^4} (u-P_0)^2 \,dx\,ds \leq \frac 1 {r_1^4} t^{1+\sigma(p)},
	 \]
	 where we used Theorem \ref{thrm.aprioribound}.

	 Turning now to the other term we have 
	 \[
	 | B(P_0 \chi_0,u-P_0)(x,t)| \lesssim \int_0^t \frac 1 {(t-s)^{1/2}}  \| P_0 \chi_0\|_{L^\I}(s) \|  u-P_0\|_{L^\I(B_0)}(s)\,ds\lesssim  t^{1/2 + \gamma/2},
	 \]
	 where we are using the fact that $P_0$ is bounded on $B_0\times [0,T]$---this is from Lemma \ref{lemma.heat} with $q=\I$---as well as the conclusion of the zeroth step to bound $\|  u-P_0\|_{L^\I(B_0)}(s)$.
	 Additionally we have  
	 \EQ{\label{ineq.helpful}
	 	| B(P_0 (1-\chi_0),u-P_0)(x,t)|&\lesssim_{u_0}  \int_0^t\int_{y\notin B_0}\frac 1 {(|x-y|+\sqrt{t-s})^4} (u-P_0) P_0 (x,s)\,dx\,ds
	 	\\&\lesssim \bigg\| \frac 1 {(|\cdot|+1)^4} \bigg\|_{L^{r_1}(0,t;L^{q_1,q_2})} \| u-P_0\|_{L^r(0,t;L^q)} \| P_0\|_{L^{\I}(0,t;L^{p,\I})},
	 }
	 where $q\in (3/2,3)$, $	r = \frac {2q} {2q-3}$ and
	 \[
	 1 = \frac 1 r_1+ \frac 1 {r}\text{; }1 =\frac 1 q_1 +\frac 1 q +\frac 1 p\text{; }1= \frac 1 {q_2}+ \frac 1 q.
	 \] 
	 We have using Proposition \ref{prop.LpinftyDecaySpacteTime} that 
	 \EQ{
	 \bigg\| \frac 1 {(|\cdot|+1)^4} \bigg\|_{L^{r_1}(0,t;L^{q_1,q_2})} \| u-P_0\|_{L^r(0,t;L^q)} \| P_0\|_{L^{\I}(0,t;L^{p,\I})} &\lesssim t^{\sigma(p) +\frac  1 {r_1}}.
	 }
	Note that $\frac 1 {r_1} = \frac 3 {2q}$ so that 
	\[
	\lim_{q\to 3/2^+} \frac 1 {r(q)} = 1.
	\]
	We conclude that, given $\delta\in (0,1)$ and taking $q>3/2$ close to $3/2$, 
	\[
		| B(P_0 (1-\chi_0),u-P_0)(x,t)| \lesssim_\delta t^{\sigma(p) +1 -\delta}.
	\]

	 Taken together the preceding estimates imply that
	 \[ { \sup_{x\in B_1}|u - P_1|(x,t) \lesssim t^{ \frac 1 2 +  \frac \gamma 2  },\quad t\in (0,T),}\]
	 and, more to the point,
	 \[
	 u = P_1 + O( t^{ \frac 1 2 + \min \{  \frac \ga 2, \frac 1 2  \}} ) +O(t^{1+\sigma(p)-\delta}) +O(t^{1+\sigma(p)}).
	 \]

	 \bigskip \noindent \underline{Second step}: 
	 We expand $u-P_1$ further as follows
	 \EQN{
	 	u-P_1 &=B(u-P_0,u-P_0) + 2B(P_0,u-P_0)
	 	\\&= B(u-P_0,(u-P_0)\chi_1)+B(u-P_0,(u-P_0)(1-\chi_1)) 
	 	\\&\quad +2B(P_0,(u-P_0) \chi_1) + 2B(P_0,(u-P_0)(1-\chi_1)).
	 } 
	 
	 First consider $2B(P_0,(u-P_0)(1-\chi_1))  $ and $B(u-P_0,(u-P_0)(1-\chi_1))$. These can be viewed as far-field contributions to the activity in $B_2\Subset B_1$.  Assuming that $x\in B_2$, we have by the same estimates on the far-field terms in the first iteration (see, e.g., \eqref{ineq.helpful}), that 
	 \[
	 2B(P_0,(u-P_0)(1-\chi_1))   \lesssim t^{1+\sigma-\delta},
	 \]
	 and 
	 \[
	 B(u-P_0,(u-P_0)(1-\chi_1))\lesssim t^{1+\sigma}.
	 \]
	 These terms are already decaying at the desired rate as $t\to 0^+$ and will therefore be left to hang out on the right-hand side of our expansion, in particular, for $x\in B_2$ we now have the expansion,
	 \EQN{
	 	u-P_1  &= B(u-P_0,(u-P_0)\chi_1) 
	 	+2B(P_0,(u-P_0) \chi_1) + O(t^{1+\sigma-\delta}).
	 } 
	 
	 We now manipulate the local terms to extract a refined asymptotic expansion. We have 
	 \EQN{
	 	&B(u-P_0,(u-P_0)\chi_1)+2B(P_0,(u-P_0) \chi_1) 
	 	\\&= B(u-P_0-B(P_0,P_0),(u-P_0)\chi_1) + B(B(P_0,P_0),(u-P_0)\chi_1) 
	 	\\&\quad +2B(P_0,(u-P_0 - B(P_0,P_0)) \chi_1) + 2B(P_0, B(P_0,P_0) \chi_1).
	 }

	 Note that $u-P_0-B(P_0,P_0)=u-P_1$ which simplifies the above expression.  Repeating this trick we obtain,
	 \EQN{ 
	 	&B(u-P_0,(u-P_0)\chi_1)+2B(P_0,(u-P_0) \chi_1) 
	 	\\&=B(u-P_1,(u-P_0)\chi_1)+B(B(P_0,P_0),(u-P_0-B(P_0,P_0))\chi_1)+B(B(P_0,P_0),B(P_0,P_0))\chi_1)  
	 	\\&\quad +2B(P_0,(u-P_1) \chi_1) + 2B(P_0,(  B(P_0,P_0) \chi_1)
	 	\\&=B(u-P_1,(u-P_0)\chi_1)+B(B(P_0,P_0),(u-P_1)\chi_1)+B(B(P_0,P_0),B(P_0,P_0))\chi_1)  
	 	\\&\quad +2B(P_0,(u-P_1) \chi_1) + 2B(P_0,  B(P_0,P_0) \chi_1)
	 }
	 Some of these terms already have decay greater than $t^{1+\sigma}$. In particular, noting that for $x\in B_0$, 
	 \[
	 |B(P_0,P_0)|(x,t)\lesssim t^{1/2}.
	 \]
	 To see this we do the usual near- and far-field split. Then,
	 \EQN{
	 	 |B(P_0,P_0)|(x,t) &\lesssim \int_0^t\int_{B_0'} \frac 1 {(t-s)^{1/2}} \|P_0\|_{L^\I}^2 + \int_0^t \int \frac C {(|x-y|-\sqrt{t-s})^4} |P_0|^2(y,s)\,dy\,ds
	 	 \\&\lesssim_{u_0,B_0, B_0'} t^{1/2} +  t .
	 }
	 Using the estimate on $B(P_0,P_0)$ we can quickly deduce the following bounds,
	 \begin{align*}
	 	& B(u-P_1,(u-P_0)\chi_1) \lesssim t^{\frac 1 2 + \frac 1 2 +\frac \gamma 2+\frac \gamma 2},
	 	\\& B(B(P_0,P_0),(u-P_1)\chi_1) \lesssim t^{\frac 1 2 +\frac 1 2 +\frac 1 2+\frac \gamma 2},
	 	\\&B(B(P_0,P_0),B(P_0,P_0))\chi_1)  \lesssim t^{\frac 1 2 +\frac 1 2 +\frac 1 2},
	 	\\&2B(P_0,(u-P_1) \chi_1) \lesssim t^{\frac 1 2 +\frac 1 2 +\frac \ga 2},
	 	\\& 2B(P_0,  B(P_0,P_0) \chi_1) \lesssim t^{\frac 1 2+\frac 1 2},
	 \end{align*}
	 which hold for all $x\in \R^3$ and $t\in (0,T)$.
	 The second and third terms are decaying faster than our target decay rate and are therefore absorbed in the $O(t^{1+\sigma-\delta})$ term on the right-hand side of our asymptotic expansion. 
	One of the remaining terms is independent of $u$. Label this $\td P_2$, i.e.,
	 \[
	 \td P_2:=  2B(P_0,  B(P_0,P_0) \chi_1).
	 \]
	 We therefore have, for $x\in B_2$ and $t\in (0,T)$, that
	 \EQN{
	 	&u - P_1 -\td P_2 = B(u-P_1,(u-P_0)\chi_1) 
	 	+2B(P_0,(u-P_1) \chi_1)  +O(t^{1+\sigma-\delta}).
	 }
	 Put differently, in $B_2$ we have the asymptotic expansion
	 \[
	 u =\underbrace{ P_1+\td P_2}_{\text{independent of $u$}} +\underbrace{ O( t^{1 +  \frac \gamma 2} ) +O(t^{1+\sigma-\delta}) }_{\text{dependent on $u$}}.
	 \]
	 It is not clear what the relationship is between $\sigma$ and $\gamma/2$. Hence we must complete one more iteration of our argument to get the desired bounds.

	 \bigskip \noindent \underline{Final step}:
	 In this step we consider the local terms on the right-hand side of the expansion at the end of the second step, namely
	 \[
	 B(u-P_1,(u-P_0)\chi_1) 
	 +2B(P_0,(u-P_1) \chi_1).\]
	 It turns out these terms are already decaying at least as fast as the target rate.
	 For example, we re-write the first of these terms as
	 \EQN{
	 	B(u-P_1,(u-P_0)\chi_1) = \underbrace{B(u-P_1,(u-P_0)(\chi_1-\chi_2))}_{\text{``far-field''}}  + \underbrace{B(u-P_1,(u-P_0) \chi_2)}_{\text{near-field}}.
	 }
	 Observe that for $(x,t)\in B_1\times [0,T]$ we have that $u-P_1$ is bounded. Hence, modifying somewhat the argument in \eqref{ineq.helpful}, the ``far-field'' contribution satisfies, for $x\in B_3\Subset B_2$,  
	 \EQN{
	 	|B(u-P_1,(u-P_0)(\chi_1-\chi_2))|(x,t)\lesssim t^{1+\frac 1 2 +\frac \ga 2 +\frac \ga 2} \lesssim t^{1+\sigma-\delta}.
	 }
	 The other term will also be decomposed by writing $\chi_1$ as $\chi_1-\chi_2 +\chi_2$.  We have for its ``far-field'' part that
	 \EQN{ \label{ineq.limiting2}
	 	&|B(P_0,(u-P_1)(\chi_1 -\chi_2)) |(x,t)   \lesssim t^{1 +\frac 1 2 +  \frac \gamma 2},
	 }
	 where, again, $x\in B_3$.
	 Note that $1+\sigma -\delta < 3/2+\gamma/2$. 
	 It is critical in the preceding estimates that there is distance between   $B_3$ and $\partial B_2$, which implies a lower bound on $|x-y|$ that shows up in the suppressed constants.
	 The above estimates show that the ``far-field'' terms can again be lumped into the $O(t^{1+\sigma-\delta})$ term in the asymptotic expansion which leads us to the revised expansion
	 \EQN{
	 	u - P_1 -\td P_2 &= B(u-P_1,(u-P_0)\chi_2) 
	  +2B(P_0,(u-P_1) \chi_2)  +O(t^{1+\sigma-\delta}).
	 }
	 It remains  to show that the local terms are also $O(t^{1+\sigma-\delta})$. Let us begin by examining $B(P_0,(u-P_1) \chi_2)$. For this we expand again as 
	 \[
	 B(P_0,(u-P_1) \chi_2) = B(P_0,(u-P_1-\td P_2) \chi_2)+ B(P_0,\td P_2 \chi_2).
	 \]
	 Both of these terms are decaying faster than our target rate, a fact we now confirm.
	 Note that for $x\in B_2$,  
	 \[
	 	|\td P_2 (x,t)|\lesssim t.
	 \] 
	 The explanation for this is almost identical to the bounding procedure for $B(P_0,P_0)$  in the preceding step.
	 Hence, 
	 \[
	| B(P_0,\td P_2 \chi_2) (x,t)|\lesssim t^{\frac 1 2 + 1}.
	 \]
	 On the other hand, 
	 \[
	 |B(P_0,(u-P_1-\td P_2) \chi_2) |\lesssim t^{\frac 1 2 + 1+\frac \ga 2}.
	 \]
	 The above estimates are true for any $x\in \R^3$.
	 For the remaining term we have, for any $x\in \R^3$,
	  \EQN{
	| B(u-P_1,(u-P_0)\chi_2)| &= |B( u-P_1-\tilde P_2 ,(u-P_0)\chi_2) + B(\td P_2 ,(u-P_0)\chi_2)|
	\\&\lesssim t^{\frac 12+1+\frac \ga 2+\frac \ga 2}+ t^{\frac 1 2 +1 +\frac \ga 2}.
	}
	Finally, let $B_\Omega=B_3$ and re-write $\td P_2$ in terms of the non-symmetric version of $B(\cdot,\cdot)$.
	This completes the proof of Theorem \ref{thrm.asympt}.

	 \begin{remark}
	 	\label{remark.generalCase} 
	 	We now discuss how this result could be generalized under the assumption that $u_0\in L^q(B)$ for $3<q<\I$. One difference is that   $t^{3/2 - 3/(2q)}$ would become another limiting rate which would need to be compared to $t^{1+\sigma-\delta}$. More iterations would also be needed because the incremental gain with each step would be $1/2-3/(2q)$ as opposed to $1/2$.   The end result would be a statement like
	 	\[
	 	u-P_\Omega = O(t^{\min\{ 	1+\sigma-\delta,3/2 - 3/(2q)	\}}),
	 	\]
	 	in $B_\Omega\times (0,T)$.
	 \end{remark}

	 We conclude this section with a proof of  Corollary \ref{cor.differentData}.
	 \begin{proof}[Proof of Corollary \ref{cor.differentData}] 
	 	In the notation of Theorem \ref{thrm.asympt} for $u_0$ we let $P_\Om(u_0) =P_\Om $    indicate the dependence on the initial data with the same convention for $P_1$ and $\td P_2$. Since, for $x\in B_\Om$ and $t\in (0,T)$,
	 	\[
	 		|u- P_1(u_0)-\td P_2(u_0) |=  O(t^{1+\sigma - \delta}),
	 	\] 
	 Therefore,
	 	\EQN{
	 	|u-v |(x,t)&\lesssim |u-P_1(u_0)-\td P_2(u_0)| +|P_1(u_0)+\td P_2(u_0)-P_1(v_0)-\td P_2(v_0)|+|v-P_1(v_0)-\td P_2(v_0)|
	 	\\&= O(t^{1+\sigma-\delta}) + |P_1(u_0)+\td P_2(u_0)-P_1(v_0)-\td P_2(v_0)|.
	 }
	 To complete the proof we just need to find an optimal estimate on $|P_1(u_0)+\td P_2(u_0)-P_1(v_0)-\td P_2(v_0)|$. 
	  We already know that for $x\in B_2$ we have
	 	\[
	 |	\td P_2(u_0) | +|	\td P_2(v_0) |\lesssim t. 
	 	\]
	 	This is the estimate asserted in the corollary so we are done with these terms.
	 	
	 	Observe that $P_1(u_0)-P_1(v_0) = e^{t\Delta} (u_0-v_0) + B(P_0(u_0),P_0(u_0)) - B(P_0(v_0),P_0(v_0))$.  Because $u_0=v_0$ in $B$ we can show as we did in \eqref{heat.helpful}  that, for $x\in B_1$,
	 	\[
			 |	e^{t\Delta} (u_0-v_0) |\lesssim t^{3/2}. 
	 	\]In fact, the above can be replaced by any power of $t$---see \eqref{helpful2}. 
	 	For $B(P_0(u_0),P_0(u_0)) - B(P_0(v_0),P_0(v_0))$, we expand this in the usual way to end up with the term
	 	\[
	 	 B(   P_0(u_0)-P_0(v_0),P_0(u_0) ),
	 	\]
	 	and another term which is treated identically and is therefore omitted. We have, 
	 	\[
	 	B(   P_0(u_0)-P_0(v_0),P_0(u_0) ) = B(   P_0(u_0)-P_0(v_0),P_0(u_0)\chi_{B_1} )+B(   P_0(u_0)-P_0(v_0),P_0(u_0) (1-\chi_{B_1})).
	 	\]
	 	For $x\in B_2$, the local term decays as desired due to the rapid decay of $ P_0(u_0)-P_0(v_0)$ which we've just proven in $B_1$. 
	 	For the far-field term, we are looking at
	 	\EQN{
	 	&|B(   P_0(u_0)-P_0(v_0),P_0(u_0) (1-\chi_{B_1}))|
	 	\\&\lesssim \int_0^t \int \frac 1 {(|x-y|+\sqrt{t-s})^4} (P_0(u_0)-P_0(v_0))P_0(u_0) (1-\chi_{B_1})\,dy\,ds.
	 	}
	 	Since there is space between the boundaries of $B_1$ and $B_2$, the spatial integral is bounded by $C (\| u_0\|_{L^{p,\I}}+\|v_0\|_{L^{p,\I}})^2$ but does not vanish as $t\to0$. Hence, the best upper bound we can get is
	 	\EQN{
	 		|B(   P_0(u_0)-P_0(v_0),P_0(u_0) (1-\chi_{B_1}))| \lesssim_{B_1,B_2,u_0,v_0} t.
	 	}

	 \end{proof}

	 \section{Time regularity}\label{sec.reg}
	 
	 In this section we establish our time regularity result, Theorem \ref{thrm.time-reg}. 
	 Our first lemma in this direction is an elementary statement about the heat equation.  Recall that 
	 \[
	 (-\Delta)^s f(x) = p.v.\int \frac { f(x)-f(y)}{|x-y|^{3+2s}}\,dy, 
	 \]
	 is the fractional Laplacian. Let $\Lambda = (-\Delta)^{\frac 1 2}$.
	 \begin{lemma}\label{lemma.heat.init}
	 	Assume $u_0\in \rev{L^{p,\I}}$ and $\Lambda^{\frac 3 2-\frac 3 p}u_0\in L^2$. Fix $p\in (2,3)$ and $T>0$. Then,  for $0<t<T$,
	 	\[
	 	\| e^{t\Delta}u_0 - u_0\|_{L^2}\lesssim t^{\frac 3 2 ( \frac 1 2 - \frac 1 p)} \| \Lambda^{\frac 3 2-\frac 3 p}u_0\|_{L^2}\lesssim_T  t^{\frac {\sigma(p)} 2} \| \Lambda^{\frac 3 2-\frac 3 p}u_0\|_{L^2}.
	 	\] 
	  Alternatively, fix $p\in (2,3]$ and assume  $u_0\in \rev{L^{p,\I}}$ and $\Lambda^{\frac 3 2-\frac 3 p}u_0\in L^{2,\I}$. Then, for a sufficiently small $\epsilon=\epsilon(\sigma(p)  )>0$ if $2<p<3$ or $\epsilon =\epsilon(\sigma(p)-)>0$  if $p=3$ and for any $t\in (0,T)$,
	 	\[
	 	\| e^{t\Delta}u_0 - u_0\|_{L^{2+\epsilon}}\lesssim_\epsilon t^{\frac 3 2 ( \frac 1 {2+\epsilon} - \frac 1 p)} \| \Lambda^{\frac 3 2-\frac 3 p}u_0\|_{L^{2,\I}} \lesssim_\epsilon
	 	 \begin{cases}
	 		t^{\sigma/2} \| \Lambda^{\frac 3 2-\frac 3 p}u_0\|_{L^{2,\I}} & \text{ if }2<p<3
	 		\\ 	t^{\sigma-/2 }\| \Lambda^{\frac 3 2-\frac 3 p}u_0\|_{L^{2,\I}} & \text{ if }p=3
	 	\end{cases}.
	 	\]
	 \end{lemma}
  In the chains of estimates above, the first bound holds for all $t\in (0,\I)$ while the second, which replaces the algebraic power of $t$ with something smaller, only holds on $(0,T)$.
 See \cite{DL} for a reference on the fractional Laplacian and the heat equation.
	 \begin{proof}
	 	We first prove the result where $\Lambda^{\frac 3 2-\frac 3 p}u_0\in L^2$.
	 	Assume $u_0\in C_c^\I$.
	 	 Observe that 
	 	 \[
	 	 e^{(t-s)\Delta}\Delta u_0 = \Lambda^s e^{(t-s)\Delta} \Lambda^{-s}\Delta u_0.
	 	 \]
	 	 Note that since $u_0$ is independent of time, \[
	 	 \partial_t ( e^{t\Delta}u_0-u_0  )-\Delta ( e^{t\Delta}u_0-u_0) = \Delta u_0.
	 	 \]
	 	 The right-hand side is a function because $u_0\in C_c^\I$. 
	 	 By Duhamel's formula we have 
	 	 \[
	 	 \| e^{t\Delta}u_0 - u_0 \|_{L^2} \leq C\int_0^t \frac 1 {(t-s)^{s/2}} \| \Lambda^{-s}\Delta u_0\|_{L^2}\,ds \leq t^{1-s/2} \| \Lambda^{-s}\Delta u_0\|_{L^2}.
	 	 \]
	 	 We take $s= \rev{\frac 1 2 +\frac 3 p}$ so that 
	 	 \[
	 	 2-s = \frac 3 2 -\frac 3 p.
	 	 \]
	 	 Then,
	 	 \[
	 	  \| e^{t\Delta}u_0 - u_0 \|_{L^2}  \leq  t^{\frac 3 4 -\frac 3 {2p}} \| \Lambda^{ \frac 3 2 -\frac 3 p} u_0\|_{L^2}.
	 	 \]
	 	 Note that  \[
	 	 \frac 3 4 -\frac 3 {2p} >\sigma/2	,
	 	 \]
	 	 for $2<p<3$. 
	 	 The above works for $u_0\in C_c^\I$. The full result follows by a density argument \rev{ in the homogeneous Sobolev space $\dot H^{\frac 3 2-\frac 3 p}$, in which $C_c^\I$ is dense since $0<\frac 3 2-\frac 3 p<\frac 3 2$.}

	 	 \bigskip We now address the case where
	 	 $\Lambda^{\frac 3 2-\frac 3 p}u_0\in L^{2,\I}$. 
	 	 By a fractional version of \cite[Corollary 2.3]{yamazaki}, we have 
	 	 \EQN{
	 	 	\| e^{t\Delta}u_0 - u_0\|_{L^{2+\epsilon}} &\leq C \int_0^t \frac {1} {(t-s)^{\frac s 2 + \frac 3 2 (  \frac 1 2 - \frac 1 {2+\epsilon} )}} \| \Lambda^{-s}\Delta u_0\|_{L^{2,\I}}\,ds
	 	 	\\&\leq C t^{1-\frac s 2-\frac 3 2 (  \frac 1 2 - \frac 1 {2+\epsilon} ) }\| \Lambda^{-s}\Delta u_0\|_{L^{2,\I}}.
	 	}
	 	 Choose $s$ as above. Note that for $p=3$,  we have $ \frac 3 4 -\frac 3 {2p} =\sigma/2$ while if $2<p<3$, 	 $\frac 3 4 -\frac 3 {2p} >\sigma/2$. In the latter case we can choose $\epsilon$ small to ensure the advertised bound holds but in the former case we must replace $t^{\sigma /2 }$ with $t^{\sigma-/2}$. Once this is done, the proof \rev{is complete. Here no density argument is needed because, for $u_0\in L^{p,\I}$, the Duhamel formula used above holds in the sense of distributions, and the time integral converges in $L^{2+\epsilon}$}.

	 \end{proof}

	 We begin with  a foundational lemma which establishes time regularity up to one derivative in the class of $L^{p,\I}$-weak solutions. 
	 
	 \begin{lemma}[Bounded time derivative]\label{lemma.firstorder} Let $B$ be an open ball and let \rev{$-\delta<t_a<t_b<\delta$}.
	 	Assume $u(x,t)$ is an $L^{p,\I}$-weak solution on $ (-\delta,\delta)\times \R^3$ \rev{(in the sense of Definition \ref{def.weakinterval})}. If $u$ is bounded on $Q= B\times \rev{(t_a,t_b)}$ and $B'$ is an open ball with $\overline {B'} \subset B$ and having the same center as $B$, then, for any \rev{$\epsilon\in(0,t_b-t_a)$},
	 	\[
	 	\partial_t u \in L^\I ( B'\times \rev{(t_a+\epsilon,t_b)} ).
	 	\]
	 \end{lemma}

	 \begin{proof}
	 	Without loss of generality assume the balls are centered at the origin. 
	 	Since $\partial_t u  = -\nb \rev{\pi} +\Delta u -u\cdot\nb u$ as distributions, it suffices to show that $-\nb \rev{\pi} +\Delta u -u\cdot\nb u$ are all bounded. Boundedness for the local terms follows from, e.g, \cite{Serrin}. The pressure is the interesting term.
	 	Let $\phi $ be smooth, radial and, for $x\in B'$, satisfy $\supp \phi(x-\cdot) \subset B$ and $\phi=1$ in a neighborhood of $x$, the radius of which is proportional to the distance from $x$ to the boundary of $B$. Recall that
	 	\[
	 (	\nb \rev{\pi} (x,t) )_k= \text{local part}+\partial_{k} \int K_{ij}(x,y) \partial_i(u_iu_j)(y)\,dy,
	 	\]
	 	where we are suppressing summation over the indices $1\leq i,j\leq 3$. The local part of the pressure, see \cite{Tsai-book}, is bounded by smoothness of $u$.
	 	We estimate the singular integral part in two cases by introducing the cut-off function $\phi$. Local terms generally look like the following
	 	\[
	 	\int K_{ij}(x,y) \phi(x-y) \partial_k \partial_i(u_iu_j)(y)  \,dy,
	 	\]
	 	while the non-local effect of the pressure is felt through the terms
	 	\[
	 	\int \partial_k[ K_{ij}(x,y) (1-  \phi(x-y)) ]  \partial_i(u_iu_j)(y)\,dy.
	 	\]
	 	The near-field integrals are all finite because $u$ is locally smooth---to prove this one uses the fact that $K_{ij}(x,y)\phi(x-y)$ is mean zero on spheres and then uses the mean value theorem to deplete the singularity of $K_{ij}$.
	 	 The far-field integrals are finite as a result to the bound,
	 	\EQN{
	 		&\int |\partial_i\partial_k[ K_{ij}(x,y) (1- \phi(x-y) )]| |(u_iu_j)|(y)  \,dy \lesssim \int \frac {\chi_{\supp [(1-\phi)(x-y)]}} {|x-y|^5} |u|^2(y)\,dy.
	 	}
	 	As the singularity is avoided in the preceding integral and $u$ has globally finite quantities in virtue of being an $L^{p,\I}$-weak solution, this term is bounded.  To elaborate, \rev{fix $s\in(-\delta,t_a+\epsilon)$ at which the restarting property in Definition \ref{def.weakinterval} holds and some $b'\in(t_b,\delta)$. For $t\in(t_a+\epsilon,t_b)$ write $u(t)=e^{(t-s)\Delta}u(s)+\td u_s(t)$, where $\td u_s(t):=u(t)-e^{(t-s)\Delta}u(s)$ is bounded in $L^2$ uniformly for $t\in (s,b')$ by condition (1) of Definition \ref{def.weak}. Let $k(z):=\chi_{\supp [1-\phi](z)}|z|^{-5}$, which belongs to $L^1\cap L^\I$ because $1-\phi$ vanishes near the origin. Using $|u|^2\leq 2|e^{(t-s)\Delta}u(s)|^2+2|\td u_s(t)|^2$, the identity $\||f|^2\|_{L^{p/2,\I}}=\|f\|_{L^{p,\I}}^2$, H\"older's inequality in Lorentz spaces \cite{ONeil} and \eqref{eqn:heat-persistance}, we have}
	 	\EQN{
	 		& \int \frac {\chi_{\supp [(1-\phi)(x-y)]}} {|x-y|^5} |u|^2(y\rev{,t})\,dy
	 		\\&\lesssim \rev{\| e^{(t-s)\Delta}u(s)\|_{L^{p,\I}}^2 \| k\|_{L^{(p/2)',1}} +  \| \td u_s(t) \|_{L^2}^2 \| k\|_{L^{\I}}
	 		\lesssim \| u(s)\|_{L^{p,\I}}^2 \| k\|_{L^{(p/2)',1}} +  \| \td u_s(t) \|_{L^2}^2 \| k\|_{L^{\I}},}
	 	}
	 	where \rev{$(p/2)'=\frac p{p-2}$ is the H\"older conjugate of $p/2>1$. Both terms are bounded uniformly in $t\in(t_a+\epsilon,t_b)$.}
	 \end{proof}

	 We now  show $\partial_t u$ is itself in a time-H\"older class. We begin with local terms as their treatment is straightforward.

	 \begin{lemma}[Time regularity of local terms] \label{lemma.localterms}	Assume $u(x,t)$ is an $L^{p,\I}$-weak solution on $ (-\delta,\delta)\times \R^3$ \rev{(in the sense of Definition \ref{def.weakinterval})} and $u$ satisfies 
	 	\[
	 	|u(x,t)|\leq \frac M {(|x|+\sqrt {\rev{|t|}})^{3/p}}.
	 	\]
	 	\rev{Fix $x\neq 0$ and suppose that $\partial_t u \in L^\I(B(x,|x|/2)\times (-\delta/2,\delta/2))$. Then, for $-\delta/4< t_1<t_2<\delta/4$,}
	 	\[
	 	|\Delta u(x,t_2)- \Delta u(x,t_1)| + 	|u\cdot\nb u(x,t_2)- u\cdot\nb u(x,t_1)|\lesssim_{|x|,M\rev{,\delta}} t_2-t_1.
	 	\]
	 	In particular, $-\Delta  u +u\cdot\nabla u  \in C_{t}^{0,1}( -\de/4,\de/4)$ for all $x\neq 0$. The same estimates hold for other partial spatial derivatives of the components of $u$ but with adjusted constants in the upper bound.
	 \end{lemma}

	 \begin{proof}
\begin{revblock}
	Fix $x\neq 0$ and write $r=|x|$. By assumption, $|u|\leq M(r/4)^{-3/p}$ on $B(x,3r/4)\times(-\delta,\delta)$ and, since $|u(y,t)|\leq M|y|^{-3/p}$, also $\sup_{-\delta<t<\delta}\|u(t)\|_{L^{p,\I}}\leq c_pM$. Below, $D^\alpha$ denotes a spatial derivative of order $|\alpha|$. The point is to show that $D^\alpha(\Delta u-u\cdot\nb u-\nb\pi)$ is bounded near $x$; we give the details in three steps.

	\smallskip\noindent\emph{Step 1: spatial regularity of $u$.} Since $u$ is a suitable weak solution which is bounded in $B(x,3r/4)\times(-\delta,\delta)$, the interior regularity theory of Serrin \cite{Serrin} implies that $u$ is smooth in the spatial variables in $Q':=B(x,r/2)\times(-\delta/2,\delta/2)$, in the sense that, for every multi-index $\alpha$, $D^\alpha u$ is bounded and continuous on $Q'$ with
	\[
	\sup_{Q'}|D^\alpha u|\leq C(\alpha,M,r,\delta).
	\]
	In particular, $D^\alpha(\Delta u-u\cdot\nb u)$ is bounded on $Q'$ for every $\alpha$.

	\smallskip\noindent\emph{Step 2: spatial regularity of $\nb\pi$.} Recall that
	\[
	\pi(z,t) = -\frac 1 3|u|^2(z,t)+p.v.\int K_{ij}(z,y)u_i(y,t)u_j(y,t)\,dy,
	\]
	see the proof of Lemma \ref{lemma.pressureHolder}. By Step 1, the first term is smooth in space on $Q'$, with bounds depending on $\alpha$, $M$, $r$ and $\delta$. For the second term, let $\eta\in C_c^\I(B(x,r/2))$ satisfy $\eta=1$ on $B(x,3r/8)$ and, for $z\in B(x,r/4)$, split
	\EQ{
	p.v.\int K_{ij}(z,y)(u_iu_j)(y,t)\,dy &= p.v.\int K_{ij}(z,y)(\eta u_iu_j)(y,t)\,dy+\int K_{ij}(z,y)((1-\eta)u_iu_j)(y,t)\,dy
	\\&=:\pi_1(z,t)+\pi_2(z,t).
	}
	Since $K_{ij}(z,y)$ is a function of $z-y$, spatial derivatives of $\pi_1$ fall on $\eta u_iu_j$, so $D^\alpha\pi_1$ is the same singular integral applied to $F:=D^\alpha(\eta u_iu_j)$, which is supported in $B(x,r/2)$ and, by Step 1, has bounded derivatives of all orders. Using that $K_{ij}$ has mean zero on spheres and the mean value theorem, as in the proof of Lemma \ref{lemma.firstorder}, we obtain $|D^\alpha\pi_1(z,t)|\lesssim \|\nb F(t)\|_{L^\I}+\|F(t)\|_{L^1}\leq C(\alpha,M,r,\delta)$. For $\pi_2$, note that $|z-y|\geq r/8$ when $z\in B(x,r/4)$ and $y\notin B(x,3r/8)$, so the kernel is smooth there and $|D^\alpha_z K_{ij}(z,y)|\lesssim_{\alpha} |z-y|^{-3-|\alpha|}\lesssim_{\alpha,r}(1+|x-y|)^{-3-|\alpha|}$. Hence, by H\"older's inequality in Lorentz spaces \cite{ONeil},
	\[
	|D^\alpha \pi_2(z,t)|\lesssim_{\alpha,r} \int \frac{|u(y,t)|^2}{(1+|x-y|)^{3+|\alpha|}}\,dy\lesssim \big\|(1+|\cdot|)^{-3-|\alpha|}\big\|_{L^{(p/2)',1}}\|u(t)\|_{L^{p,\I}}^2\lesssim_{\alpha,r} M^2,
	\]
	where $(p/2)'=\frac p{p-2}\in[3,\I)$ is the H\"older conjugate of $p/2$. We conclude that, for every $\alpha$,
	\[
	\sup_{B(x,r/4)\times(-\delta/2,\delta/2)}|D^\alpha\nb\pi|\leq C(\alpha,M,r,\delta).
	\]

	\smallskip\noindent\emph{Step 3: time regularity.} Since $\partial_t u=\Delta u-u\cdot\nb u-\nb\pi$ in the sense of distributions, for every $\alpha$ we have
	\[
	\partial_t D^\alpha u = D^\alpha(\Delta u-u\cdot\nb u-\nb \pi)
	\]
	in the sense of distributions on $B(x,r/4)\times(-\delta/2,\delta/2)$, and, by Steps 1 and 2, the right-hand side is bounded there by a constant $C(\alpha,M,r,\delta)$.\footnote{It is worth pointing out that if we wanted to bound $\partial_t^2 u$ then we would need to contend with $ \partial_t (\Delta u -u\cdot\nabla u-\nabla \pi)$. This is not necessarily finite because the time derivative would be passed onto the non-local part of the pressure. This is why first order time regularity is easy to obtain while higher order time regularity is not. } 
	To pass from this distributional identity to a pointwise estimate, write $Q'':=B(x,r/4)\times(-\delta/2,\delta/2)$, $g:=D^\alpha u$ and $F:=D^\alpha(\Delta u-u\cdot\nb u-\nb \pi)$, so that $\partial_t g=F$ in $\mathcal D'(Q'')$ and $\|F\|_{L^\I(Q'')}\leq C(\alpha,M,r,\delta)$. Let $\rho_\epsilon$ be a standard mollifier in $(y,t)$ and let $g_\epsilon:=g*\rho_\epsilon$, which is smooth on $Q''_\epsilon:=\{z\in Q'':\operatorname{dist}(z,\partial Q'')>\epsilon\}$. Since convolution commutes with distributional derivatives, $\partial_t g_\epsilon=F*\rho_\epsilon$ on $Q''_\epsilon$, and hence $|\partial_t g_\epsilon|\leq \|F\|_{L^\I(Q'')}$ there. Fix $y\in B(x,r/4)$ and $-\delta/2<t_1<t_2<\delta/2$. For $\epsilon$ sufficiently small, the segment $\{y\}\times[t_1,t_2]$ lies in $Q''_\epsilon$, and the fundamental theorem of calculus gives
	\[
	|g_\epsilon(y,t_2)-g_\epsilon(y,t_1)|\leq \|F\|_{L^\I(Q'')}\,(t_2-t_1).
	\]
	Since $g$ is continuous on $Q''$ by Step 1, $g_\epsilon\to g$ locally uniformly on $Q''$ as $\epsilon\to0$. Letting $\epsilon\to0$, we conclude that, for every $y\in B(x,r/4)$ and $-\delta/2<t_1<t_2<\delta/2$,
	\[
	|D^\alpha u(y,t_2)-D^\alpha u(y,t_1)|\leq C(\alpha,M,r,\delta)(t_2-t_1).
	\]
	Taking $|\alpha|=2$ gives the claimed estimate for $\Delta u$. For the nonlinear term we write
	\[
	u\cdot\nb u(y,t_2)-u\cdot\nb u(y,t_1)=\big(u(y,t_2)-u(y,t_1)\big)\cdot\nb u(y,t_2)+u(y,t_1)\cdot\big(\nb u(y,t_2)-\nb u(y,t_1)\big),
	\]
	and use the preceding estimate with $|\alpha|\leq 1$ together with Step 1. The same argument applies to $D^\alpha u(x,\cdot)$ for any multi-index $\alpha$, with a constant depending on $|\alpha|$. Since we only ever consider $|\alpha|\leq 2$, we can take this to be universal.
\end{revblock}
	 \end{proof}

	 We next investigate the time regularity of the pressure gradient.
	 
	 \begin{lemma}[Time regularity of the pressure gradient] 	\label{lemma.pressureHolder}Fix $p\in (2,3]$. Assume $u(x,t)$ is an $L^{p,\I}$-weak solution on $ (-\delta,\delta)\times \R^3$ \rev{(in the sense of Definition \ref{def.weakinterval})}.
	 	Suppose that $\partial_t u \in L^\I(B(x,|x|/2)\times \rev{(-\delta/2,\delta/2)})$ and $u$ satisfies the assumptions of Theorem \ref{thrm.time-reg}.
	  Fix $-\delta/4< t_1<t_2<\delta/4$ and $x\neq 0$. Then,
	 	\[
	 	|\nb \rev{\pi}(x,t_2)- \nb \rev{\pi}(x,t_1)|\lesssim_{|x|,M\rev{,\delta}}  \begin{cases}
	 		 (t_2-t_1)^{\sigma/2}&2<p<3
	 		 \\(t_2-t_1)^{\sigma-/2}&p=3
	 	\end{cases}.
	 	\] 
	 \end{lemma}

	 Note that, here, the regularity is $C_{t}^{0,\sigma/2}$ or $C_{t}^{0,\sigma-/2}$ whereas for the local terms it was $ C^{0,1}_t$. This bears witness to the fact that the pressure is the limiting term insofar as time regularity is concerned.  
	 
	 \begin{proof} Fix $x\neq 0$. Recall that 
	 	\[
	 	\rev{\pi}(x,t) =  -\frac 1 3 |u|^2(x) + p.v. \int K_{ij}(x,y) u_i(y,t)u_j (y,t)\,dy, 
	 	\]
	 	where $K$ is the matrix kernel of a Calder\'on-Zygmund operator with entries $(K_{ij})$---see, e.g., \cite{Tsai-book}---and we are suppressing summation over $1\leq i,j\leq 3$. Hence,
	 	\[
	 	\partial_{x_k} \rev{\pi}(x,t)=-\frac 2 3 u_i\partial_{x_k}u_i + p.v. \int K_{ij}(x,y) \partial_{y_k}(u_i(y,t)u_j (y,t))\,dy.
	 	\]
	 	 Lemma \ref{lemma.localterms} applies to $u_i\partial_{x_k}u_i $ indicating it has the desired regularity. We will therefore not consider this term again in this proof.
	 	We re-write the singular integral in terms of near- and far-field parts by introducing a cut-off function $\phi$ so that $\phi=1$ on $B(0,1/4)$, $\phi \in C_c^\I(B(0,1/2))$. We require $\phi$ to also be radial. Let 
	 	\[
	 	I_{\text{near}}=p.v. \int \phi((x-y)/|x|) K_{ij}(x,y) \partial_{y_k}(u_i(y,t)u_j (y,t))\,dy,
	 	\]
	 	and 
	 	\[
	 	I_{\text{far}}=\int   \big(1-\phi((x-y)/|x|) \big) K_{ij}(x,y) \partial_{y_k}(u_i(y,t)u_j (y,t))\,dy.
	 	\]
	 	Then 
	 	\[p.v. \int K_{ij}(x,y) \partial_{y_k}(u_i(y,t)u_j (y,t))\,dy = I_{\text{near}}+I_{\text{far}}.\]
	 	We re-write $I_{\text{far}}$ as the sum of
	 	\[
	 	I_{\text{far}}^i  = \int   \partial_{y_k} \big[\big(1-\phi((x-y)/|x|) \big) \big] K_{ij}(x,y)u_i(y,t)u_j (y,t)\,dy,
	 	\]
	 	and
	 	\[
	 	I_{\text{far}}^{ii} = \int    \big[\big(1-\phi((x-y)/|x|) \big) \partial_{y_k}K_{ij}(x,y)u_i(y,t)u_j (y,t)\,dy.
	 	\]
	 	The first of these is effectively a local term because $u_i$ and $u_j$ are only evaluated at points $y$ near to, but not too near to, $x$. The latter is the only genuinely far-field term. To bound it, we will use properties of $L^{p,\I}$-weak solutions.
	 	
	 	\medskip \noindent
	 	\underline{Far-field estimates:} 
	 	Let us consider 
	 	\EQN{
	 		&	I_{\text{far}}^{ii}(x,t_2)-		I_{\text{far}}^{ii}(x,t_1) 
	 		\\&= \int    \big[\big(1-\phi((x-y)/|x|) \big) \partial_{y_k}K_{ij}(x,y)  \big[ u_i(y,t_2)u_j (y,t_2) -  u_i(y,t_1)u_j (y,t_1)\big]\,dy
	 		\\&= \int    \big[\big(1-\phi((x-y)/|x|) \big) \partial_{y_k}K_{ij}(x,y)  \big( u_i(y,t_2)- u_i(y,t_1)\big)u_j (y,t_2)  \,dy
	 		\\&\quad + \int    \big[\big(1-\phi((x-y)/|x|) \big) \partial_{y_k}K_{ij}(x,y)  \big( u_j(y,t_2)- u_j(y,t_1)\big)u_i(y,t_1)  \,dy.
	 	}
	 	\rev{Let $\tau=t_2-t_1$. Let $\Sigma\subset(-\delta,\delta)$ be the set of full measure of times $s$ at which the restarting property in Definition \ref{def.weakinterval} holds, and choose $t_0\in\Sigma\cap(t_1-\frac 3 2\tau,\,t_1-\tau)$, which is possible since this interval has positive length. Since $t_1>-\delta/4$ and $\tau<\delta/2$, we have $t_0\in(-\delta,\delta/4)$. Let $\tau_1:=t_1-t_0$ and $\tau_2:=t_2-t_0=\tau_1+\tau$, so that $\tau<\tau_1<\frac 3 2\tau$ and $2\tau<\tau_2<\frac 5 2\tau$.}
	 	Our strategy is to add and subtract $e^{\rev{\tau_1}\Delta}u(t_0)$ and $e^{\rev{\tau_2}\Delta}u(t_0)$ so that we can use the \textit{a priori} bound from \rev{Corollary \ref{cor:apriori-bound}, restarted at the time $t_0$ as in Lemma \ref{lemma.restart}}.
	 In particular, we have
	 	\EQN{
	 		&\int    \big[\big(1-\phi((x-y)/|x|) \big) \partial_{y_k}K_{ij}(x,y)  \big( u_i(y,t_2)- u_i(y,t_1)\big)u_j (y,t_2)  \,dy 
	 		\\&= \int \big[\big(1-\phi((x-y)/|x|) \big) \partial_{y_k}K_{ij}(x,y)  \big( u_i(y,t_2)- e^{\rev{\tau_2}\Delta} u_i(y,t_0)\big)u_j (y,t_2)  \,dy
	 		\\&\quad + \int    \big[\big(1-\phi((x-y)/|x|) \big) \partial_{y_k}K_{ij}(x,y)  	\big( e^{\rev{\tau_2}\Delta} u_i(y,t_0)	-e^{\rev{\tau_1}\Delta} u_i(y,t_0)\big)u_j (y,t_2)  \,dy 
	 		\\&\quad + \int    \big[\big(1-\phi((x-y)/|x|) \big) \partial_{y_k}K_{ij}(x,y)  	\big(e^{\rev{\tau_1}\Delta} u_i(y,t_0)	-	u_i(y,t_1)\big)\rev{u_j (y,t_2)  \,dy}.
	 	}
\begin{revblock}
	 	By \eqref{ineq.assumption}, $|u(y,t)|\leq M|y|^{-3/p}$, so that $\sup_{-\delta<t<\delta}\|u(t)\|_{L^{p,\I}}\leq c_pM$. Since $t_0\in\Sigma$, Corollary \ref{cor:apriori-bound}, applied to the translate of $u$ which starts at time $t_0$ (equivalently, \eqref{ineq.restart}), gives, for $h\in\{\tau_1,\tau_2\}$,
	 	\[
	 	\| u(t_0+h)- e^{h\Delta} u(t_0) \|_{L^2}^2\leq C_p\left[(c_pM)^{\frac {2p}{4-p}}h^{\sigma}+h^{\frac 1 2}\right]\leq C(p,M,\delta)\,h^{\sigma},
	 	\]
	 	where $\sigma=\sigma(p)$ and we used $\sigma\leq 1/2$ and $h<2\delta$. Since $\tau_1<\tau_2<\frac 5 2\tau$, it follows that
	 	\[
	 	\| u(t_2)- e^{\tau_2\Delta} u(t_0) \|_{L^2}+\| e^{\tau_1\Delta} u(t_0)-u(t_1) \|_{L^2}\lesssim_{M,\delta} \tau^{\sigma/2} = (t_2-t_1)^{\sigma/2}.
	 	\]
	 	We emphasize that $t_0$ is chosen in $\Sigma$, and that the only requirement on it is that $t_1-t_0$ and $t_2-t_0$ are comparable to $t_2-t_1$. In particular, the argument never restarts the solution at a prescribed time, such as the singular time $t=0$. We therefore have
\end{revblock}
	 	\EQN{
	 		& \int \big[\big(1-\phi((x-y)/|x|) \big) \partial_{y_k}K_{ij}(x,y)  \big( u_i(y,t_2)- e^{\rev{\tau_2}\Delta} u_i(y,t_0)\big)u_j (y,t_2)  \,dy 
	 		\\& \lesssim_M (t_2-t_1)^{\sigma /2}\|	 1-\phi((x-y)/|x|) \big) \partial_{y_k}K_{ij}(x,y)  	\|_{L_y^{r,2}(\R^3)} \| u\|_{L^\I(-\delta,\delta;L^{p,\I}(\R^3))},
	 	}
	 	where 
	 	\[
	 	1 = \frac 1 2 +\frac 1 p +\frac 1 r.
	 	\]
	 	Note that $p>2$ so   $r$ can be suitably chosen.   We have due to the support of $1-\phi$ that
	 	\[
	 	( 1-\phi((x-y)/|x|) ) \partial_{y_k}K(x,y) \lesssim_{|x|} \frac 1 {(1+|y|)^4} \in L^\I \cap L^1 \subset L^{r,2}(\R^3),
	 	\]
	 	and thus obtain 
	 	\EQN{
	 		& \int \big[\big(1-\phi((x-y)/|x|) \big) \partial_{y_k}K_{ij}(x,y)  \big( u_i(y,t_2)- e^{\rev{\tau_2}\Delta} u_i(y,t_0)\big)u_j (y,t_2)  \,dy 
	 		\lesssim_{M,|x|} (t_2-t_1)^{\sigma/2}.
	 	}
	 	An identical bound holds for the term involving   $e^{\rev{\tau_1}\Delta} u_i(y,t_0)	-	u_i(y,t_1)$. 
	 	
	 	The remaining term in $	I_{\text{far}}^{ii}(x,t_2)-		I_{\text{far}}^{ii}(x,t_1) $ is 
	 	\EQN{
	 		&\int    \big[\big(1-\phi((x-y)/|x|) \big) \partial_{y_k}K_{ij}(x,y)  	\big( e^{\rev{\tau_2}\Delta} u_i(y,t_0)	-e^{\rev{\tau_1}\Delta} u_i(y,t_0)\big)u_j (y,t_2)  \,dy .
	 	}
\begin{revblock}
	 	Note that $e^{\tau_2\Delta} u(t_0)-e^{\tau_1\Delta} u(t_0)=e^{\tau_1\Delta}\big(e^{\tau\Delta}u(t_0)-u(t_0)\big)$, since $\tau_2-\tau_1=\tau$, and that $e^{\tau_1\Delta}$ is a contraction on $L^q$ for every $q\in[1,\I]$. If $2<p<3$ and the assumption involving $\dot H^{\frac 3 2-\frac 3 p}$ in Theorem \ref{thrm.time-reg} holds, then, using the first part of Lemma \ref{lemma.heat.init} with $u_0=u(t_0)\in L^{p,\I}$ and $T=\delta$, we have
	 	\[
	 	\| e^{\tau_2\Delta} u_i(y,t_0)	-e^{\tau_1\Delta} u_i(y,t_0)\|_{L^2}\leq \|e^{\tau\Delta}u(t_0)-u(t_0)\|_{L^2}\lesssim_\delta (t_2-t_1)^{\sigma/2} M.
	 	\]
	 	So,
	 	\EQN{
	 		&\int    \big[\big(1-\phi((x-y)/|x|) \big) \partial_{y_k}K_{ij}(x,y)  	\big( e^{\tau_2\Delta} u_i(y,t_0)	-e^{\tau_1\Delta} u_i(y,t_0)\big)u_j (y,t_2)  \,dy 
	 		\\&\leq C (t_2-t_1)^{\sigma/2} M^2.
	 	}
	 	If instead the assumption involving $L^{2,\I}$ in Theorem \ref{thrm.time-reg} holds (which is the only option when $p=3$), we use the alternative case in Lemma \ref{lemma.heat.init}. It gives an estimate in $L^{2+\epsilon}$ rather than in $L^2$, namely,
	 	\[
	 	\| e^{\tau_2\Delta} u_i(y,t_0)	-e^{\tau_1\Delta} u_i(y,t_0)\|_{L^{2+\epsilon}}\leq \|e^{\tau\Delta}u(t_0)-u(t_0)\|_{L^{2+\epsilon}}\lesssim_{\epsilon,\delta} \begin{cases} (t_2-t_1)^{\sigma/2} M & 2<p<3\\ (t_2-t_1)^{\sigma-/2} M & p=3\end{cases},
	 	\]
	 	for a sufficiently small $\epsilon>0$ as in Lemma \ref{lemma.heat.init}. We then apply H\"older's inequality with the exponents $2+\epsilon$, $p$ and $r_\epsilon$, where $1=\frac 1 {2+\epsilon}+\frac 1 p+\frac 1 {r_\epsilon}$. Since $p>2$, we have $r_\epsilon\in(1,\I)$ for $\epsilon$ small, and $( 1-\phi((x-y)/|x|) ) \partial_{y_k}K_{ij}(x,y)$ belongs to $L^1\cap L^\I$, hence to the Lorentz space $L^{r_\epsilon,s}$ for every $s\in[1,\I]$, as before. This leads to
	 	\EQN{
	 		&\int    \big[\big(1-\phi((x-y)/|x|) \big) \partial_{y_k}K_{ij}(x,y)  	\big( e^{\tau_2\Delta} u_i(y,t_0)	-e^{\tau_1\Delta} u_i(y,t_0)\big)u_j (y,t_2)  \,dy 
	 		\\&\leq C \begin{cases} (t_2-t_1)^{\sigma/2} M^2 & 2<p<3\\ (t_2-t_1)^{\sigma-/2} M^2 & p=3\end{cases}.
	 	}
\end{revblock}
	 	 This concludes our estimate on $|I_{\text{far}}^{ii}(x,t_2)-		I_{\text{far}}^{ii}(x,t_1)|$.

	 	We now bound $	I_{\text{far}}^{i}(t_2)-	I_{\text{far}}^{i}(t_1)$,
	 	which amounts to bounding the following integral
	 	\EQN{
	 		\int   \partial_{y_k} \big[\big(1-\phi((x-y)/|x|) \big) \big] K_{ij}(x,y)\big[ u_i(y,t_2)u_j (y,t_2)-u_i(y,t_1)u_j (y,t_1)   \big]\,dy.   
	 	}
	 	Observe that $ \partial_{y_k} \big[\big(1-\phi((x-y)/|x|) \big) \big]$ is supported away from any singularities in either $K_{ij}$ or $u$ and \rev{in $B(x,|x|/2)$}. Furthermore, the difference of products can be expanded using the standard trick to obtain terms like $u_i(y,t_2)(  u_j(y,t_2)-u_j(y,t_1))$. The first factor is bounded using the assumption \eqref{ineq.assumption} while the second is bounded using \rev{the hypothesis on $\partial_t u$} and the mean value theorem in the time variable.  Using these observations we obtain
	 	\EQN{
	 		|	I_{\text{far}}^{i}(t_2)-	I_{\text{far}}^{i}(t_1)  |\lesssim_{M,x } (t_2-t_1)\leq  (t_2-t_1)^{\sigma/2},
	 	}
	 	since $\sigma/2<1$.
	 	
	 	\medskip\noindent\underline{Near-field estimates}:
	 	We still need to bound $	I_{\text{near}}$. This is complicated by the critical order singularity in $K$ which means we must use  the mean value theorem in both  the time and space variables. We need to estimate the following,
	 	\begin{align*}
	 	p.v. \int \phi((x-y)/|x|) K_{ij}(x,y) \bigg(\partial_{y_k}(u_i(y,t_2)u_j (y,t_2))-\partial_{y_k}(u_i(y,t_1)u_j (y,t_1))  \bigg)\,dy.
	 	\end{align*}
	 	This breaks into a number of symmetric cases after we use the product rule. We only consider one such case, namely
	 	\begin{align*}
	 		p.v. \int \phi((x-y)/|x|) K_{ij}(x,y) \bigg(u_i(y,t_2)\partial_{y_k}u_j (y,t_2)- u_i(y,t_1)\partial_{y_k}u_j (y,t_1)  \bigg)\,dy.
	 	\end{align*}
	  Applying a standard trick, this reduces to more terms which are treated the same, including, e.g., the term
	\begin{align*}
		p.v. \int \phi((x-y)/|x|) K_{ij}(x,y) \big(u_i(y,t_2)-u_i(y,t_1)\big)\partial_{y_k}u_j (y,t_2) \,dy.
	\end{align*}
	  To be clear, there are more terms  than this but they are symmetric so we omit them. The kernel vanishes on spheres centered at $x$ and, therefore, 
	  	\begin{align*}
	  	&p.v. \int \phi((x-y)/|x|) K_{ij}(x,y) \big(u_i(y,t_2)-u_i(y,t_1)\big)\partial_{y_k}u_j (y,t_2) \,dy
	  	\\&= 		p.v. \int \phi((x-y)/|x|) K_{ij}(x,y)
	  	\\&\qquad\cdot \bigg( \big(\underbrace{u_i(y,t_2)-u_i(y,t_1)}_{=: v(y)}\big)\partial_{y_k}u_j (y,t_2) - \big(\underbrace{u_i(x,t_2)-u_i(x,t_1)}_{=:v(x)}\big)\partial_{x_k}u_j (x,t_2)  \bigg) \,dy.
	  \end{align*}
	 	 We have 
	 	\[
	 	v(y)\partial_{y_k}u_j (y,t_2)- v(x)\partial_{x_k}u_j (x,t_2)= 	v(y)\big(\partial_{y_k}u_j (y,t_2)-\partial_{x_k}u_j (x,t_2) \big)+ (v(y)- v(x))\partial_{x_k}u_j (x,t_2).
	 	\]
	 	\rev{The support of $\phi((x-\cdot)/|x| )$ is contained in $B(x,|x|/2)$. There, by the hypothesis on $\partial_t u$,
	 	\[
	 	|v(y)|\leq \|\partial_t u\|_{L^\I(B(x,|x|/2)\times(-\delta/2,\delta/2))}\,(t_2-t_1),
	 	\]
	 	and, by Step 1 of the proof of Lemma \ref{lemma.localterms}, the spatial derivatives of $u$ are bounded on $B(x,|x|/2)\times(-\delta/2,\delta/2)$.} Hence
	 	\[
	 	|v(y)\cdot (\partial_{y_k}u_j (y,t_2)-\partial_{x_k}u_j (x,t_2)) |\lesssim (t_2-t_1)|x-y|,
	 	\]
	 	where we used the mean value theorem in the $x$ variable. 
	 	For the other term, applying the mean value theorem in the space variable to $v(y)-v(x)$ gives
	 	\[
	 	|(v(y)-v(x))\partial_{x_k}u_j (x,t_2)|\lesssim |\partial_{x_k}u_j (x,t_2)| |x-y| \sup_{z\in [x,y]} |\nabla v(z)|,
	 	\]
	 	where $z\in [x,y]$ just means that $z$ is on the line segment connecting $x$ to $y$. \rev{Such points satisfy $z\in B(x,|x|/2)$, so that $|z|>|x|/2$, and $\nabla v(z) = \nabla u_i (z,t_2)- \nabla u_i (z,t_1)$. The proof of Lemma \ref{lemma.localterms} uses the choice $r=|x|$ only through the bound $|u|\leq M(r/4)^{-3/p}$ on $B(x,3r/4)\times(-\delta,\delta)$, and this bound follows from \eqref{ineq.assumption} as soon as $r\leq |x|$. Since $|z|>|x|/2$, we may therefore run that proof with $x$ replaced by $z$ and $r=|x|/2$, and Step 3 with $|\alpha|=1$ gives $|\nabla v(z) |\leq C(M,|x|,\delta)(t_2-t_1)$} and, therefore,
	 	\[
	 	|(v(y)-v(x))\cdot\nb u(x)|\lesssim |\nb u(x)| |x-y|(t_2-t_1)
	 	\]
	 	Because we are able to extract both powers of $|x-y|$ and $(t_2-t_1)$, we are able to deplete the singularity in the kernel, subsequently obtaining the bound,
	 	\[
	 	| I_{\text{near}}(x,t_2)-	I_{\text{near}}(x,t_1)|\lesssim_{x,M} (t_2-t_1) \lesssim_\delta (t_2-t_1)^{\sigma/2}.
	 	\]

	 \end{proof}

	 We can now complete the proof of our  time-regularity result by combining the preceding two lemmas.
	 
	 \begin{proof}[Proof of Theorem \ref{thrm.time-reg}]
	 	\rev{Fix $x\neq0$. By \eqref{ineq.assumption}, $u$ is bounded on $B(x,3|x|/4)\times(-\delta,\delta)$ by $M(|x|/4)^{-3/p}$. Hence Lemma \ref{lemma.firstorder}, applied with $B=B(x,3|x|/4)$, $B'=B(x,|x|/2)$, $(t_a,t_b)=(-3\delta/4,3\delta/4)$ and $\epsilon=\delta/4$, shows that $\partial_t u\in L^\I(B(x,|x|/2)\times(-\delta/2,\delta/2))$, which is the hypothesis on $\partial_t u$ in Lemmas \ref{lemma.localterms} and \ref{lemma.pressureHolder}. Since $\partial_t u(x,\cdot)=(\Delta u-u\cdot\nb u-\nb\pi)(x,\cdot)$, the theorem now follows from these two lemmas, which give the time regularity of $\Delta u-u\cdot\nb u$ and of $\nb\pi$, respectively, together with the bounds on $u(x,\cdot)$ and $\partial_tu(x,\cdot)$ just mentioned.}
	 \end{proof}
	 
		 \section*{Declarations}

	\textbf{Funding:} The research of Z.~Bradshaw was supported in part by the  NSF via grant DMS-2307097.
	
\noindent 	\textbf{Interests:} The authors have no competing interests to declare that are relevant to the content of this article.
	 
	 \noindent 	\textbf{Data availability statement:} No datasets were generated or analyzed during the current study.
	 
	 \rev{\noindent\textbf{Statement on the use of AI.} We used Claude Opus 5.5 to assist with implementing several revisions suggested by the referees which included fixing one error and closing one gap. This was done under the close guidance of the authors. In particular, every edit that was suggested by AI was verified independently by the authors who take full responsibility for the manuscript.} 

\end{document}